\documentclass{article}
\usepackage{graphicx} % Required for inserting images

\title{Free Left Distributive Algebras and a Canonical Extension}
\author{Scott Cramer \and Meng-Che ``Turbo'' Ho\thanks{The second author acknowledges support by the National Science Foundation under Award No.~2054558 and No.~2555186.} \and Sheila K. Miller Edwards \and Nam Trang\thanks{The fourth author acknowledges support by the National Science Foundation under Career Award No. 1945592.}\ \thanks{The project began at a SQuaRE at the American Institute of Mathematics. The authors thank AIM for providing a supportive and mathematically rich environment.}}
\date{}%June 26, 2025}

\usepackage{amsmath}

\usepackage{amsthm}
\usepackage{amssymb}

\usepackage{url}
\usepackage{hyperref}
\usepackage{todonotes}
\usepackage{comment}
\usepackage{enumerate}

\theoremstyle{plain}
\newtheorem{theorem}{Theorem}
\newtheorem{thm}[theorem]{Theorem}
\newtheorem{prop}[theorem]{Proposition}
\newtheorem*{theorem*}{Theorem}

\newtheorem{lemma}[theorem]{Lemma}
\newtheorem{lem}[theorem]{Lemma}
\newtheorem{cor}[theorem]{Corollary}
\newtheorem{corollary}[theorem]{Corollary}
\newtheorem{question}[theorem]{Question}
\newtheorem{conjecture}[theorem]{Conjecture}
\newtheorem{remark}[theorem]{Remark}
\newtheorem{rmk}[theorem]{Remark}

\theoremstyle{definition}
\newtheorem{definition}[theorem]{Definition}
\newtheorem{dfn}[theorem]{Definition}

\theoremstyle{remark}

\newcommand{\rng}{\operatorname{rng}}
\newcommand{\Acal}{\mathcal{A}}
\newcommand{\Ccal}{\mathcal{C}}

\newcommand{\isom}{\cong}
\newcommand{\rest}{\restriction}

\newcommand{\la}{\left <}
\newcommand{\ra}{\right >}
\newcommand{\Ecal}{\mathcal{E}}

\newcommand{\crit}{\text{crit }}
\newcommand{\irng}{\text{irng }}
\newcommand{\sseq}{\subseteq}

\newcommand{\vlm}{V_{\lambda+1}}

\begin{document}

\maketitle

\begin{abstract}
    Assuming a large cardinal hypothesis, Laver gave a representation of the monogenerated free left distributive algebra (LDA) using elementary embeddings and used this representation to prove many algebraic results. Some of these results were later proved by Dehornoy in ZFC, without the large cardinal hypotheses. However, there is an important algebraic result whose consistency strength is unknown \cite{Laver:1995, DoughertyJech:1997}. Recent results \cite{BrookeTaylorCramerMiller2024} extend the connection between elementary embeddings of set theory and free LDAs to the many-generated case.

    Assuming large cardinals, we prove two results. First, we prove that finitely-generated free LDAs with distinct numbers of generators are $\Sigma_1$-elementarily equivalent but not $\Sigma_2$-elementarily equivalent. We also prove a partial structural analogue to Laver's representation of LDAs. We construct an extension of the monogenerated free LDA 
    where application by any fixed element is an elementary embedding of LDAs.
   
    We argue that this extension is canonical by demonstrating homogeneity and universality properties. These results also provide additional examples of algebraic properties provable from large cardinals without known proofs from the standard axioms of set theory.
\end{abstract}

\section{Introduction}\label{Introduction}
Left distributive algebras are those in which multiplication on the left by an element of the algebra is a homomorphism: In a left distributive algebra with underlying set $A$ and binary operation $\cdot$, for every element $a, b$, and $c$, the equality $a \cdot (b \cdot c) = (a \cdot b) \cdot (a \cdot c)$ holds. Many well-known mathematical operations, including group conjugation and the weighted mean, are left distributive. In each of these cases, the operation is also idempotent ($x \cdot x = x$ for every element $x$ in the algebra), hence the algebras are not free.

Beginning in the 1980s, a deep connection was discovered between left distributive algebras and large cardinals, logical axioms the existence of which implies the consistency of the usual axioms of set theory, ZFC (and also many statements independent of ZFC). By G\"odel's Incompleteness Theorems, the consistency of large cardinal axioms does not follow from ZFC. Large cardinal axioms are linearly ordered by consistency strength,\footnote{All natural large cardinal axioms of interest to set theorists are linearly ordered.} and the stronger large cardinal axioms can typically be formulated to assert the existence of elementary embeddings with specific closure properties. Under one such large cardinal hypothesis called a \emph{rank-to-rank} embedding, Richard Laver \cite{Laver:1992} proved that the algebra generated by the closure of a single rank-to-rank embedding under the \emph{application} operation is linearly ordered by the iterated left subterm relation $<_L$  and isomorphic to the free left distributive algebra on one generator, $\Acal$. (For details, see Section \ref{Background}.) In so doing, Laver proved that the word problem for the left distributive law is decidable.

Patrick Dehornoy subsequently proved that the decidability of the word problem is a purely algebraic result not requiring large cardinals \cite{Dehornoy:1994}. In particular, he defined an operation on elements of Artin's braid group on infinitely many strands $B_{\infty}$ and showed in ZFC that the closure of any element of $B_{\infty}$ under that operation is isomorphic to the free left distributive algebra on one generator, $\mathcal{A}$. (See Section \ref{Background}.) 

Work to generalize and extend these early results and to better understand the relationship between large cardinals and left distributive algebras has continued since the 1990s. Also using Artin's braid group, David Larue significantly simplified Dehornoy's proof \cite{Larue:1994} and constructed free left distributive algebras on $n$ generators for each finite $n$, now called the Larue groups \cite{LarueThesis:1994}. Though it seemed that an example of a many-generated free left distributive algebra should also exist in the context of rank-to-rank embeddings, the demonstration of such an algebra proved elusive until quite recently \cite{BrookeTaylorCramerMiller2024, BTM:2026}. Many questions from the 1990s remain open, including one about the consistency strength of a theorem about a family of finite algebras known as the Laver tables.

Perhaps the central question of the theory of left distributive algebras is whether there are statements in the theory of left distributive algebras that contain large cardinal strength---namely, statements the assertion of which implies the existence of a large cardinal---or whether all structural properties of left distributive algebras can be proven from the usual axioms of set theory (or from considerably weaker theories). Our overarching goal is to understand which of these two possibilities holds. This paper contributes two further structural results on left distributive algebras derivable from large cardinals. First, we show that finitely-generated free left distributive algebras on distinct numbers of generators are $\Sigma_1$-elementarily equivalent but not $\Sigma_2$-elementarily equivalent; second, we define a left distributive algebra $\Ccal_1$ that preserves more properties of algebras of rank-to-rank embeddings than does the monogenerated free left distributive algebra $\Acal$. Whether it is possible to prove either of the two main results of this paper without large cardinal hypotheses is open.

The first of our two main results shows, under a large cardinal hypothesis, that, in spite of the qualitative differences between working with one- and many-generated free left distributive algebras, $\Sigma_1$ formulas cannot distinguish between finitely-generated, free left distributive algebras with different numbers of generators, while $\Sigma_2$ formulas can.

\begin{theorem*}[A]
    Under appropriate large cardinal assumptions, for any two distinct positive integers $m$ and $n$, the free left distributive algebras on $m$ and $n$ generators, $\mathcal{A}_m$ and $\mathcal{A}_n$, are $\Sigma_1$-elementarily equivalent but not $\Sigma_2$-elementarily equivalent.
\end{theorem*}

Laver's result that the algebra $\Acal_{\{j\}}$ generated by closing a single rank-to-rank elementary embedding $j$ under the application operation is free can be viewed as asserting that $\Acal_{\{j\}}$ always embeds into a simply defined algebraic structure (namely, the free left distributive algebra on a single generator). His result suggests a natural question: Does the closure of every finite set of rank-to-rank embeddings generate an algebra that embeds into a simply definable structure (such as the free left distributive algebra on the appropriate number of generators)?

The algebras generated by finite collections of rank-to-rank embeddings are more complex than $\Acal_{\{j\}}$ and so, too, is the answer to this question.

Much of the time, it is not clear what the algebra generated by two embeddings might look like. However, under special circumstances, it is possible to find two embeddings that do generate an algebra isomorphic to the free, two-generated left distributive algebra $\Acal_2$ \cite{BrookeTaylorCramerMiller2024}. Indeed, for any cardinal $\kappa$, $1 \leq \kappa \leq 2^{\aleph_0}$, we can find a collection of elementary embeddings that generate the free left distributive algebra $\mathcal{A}_{\kappa}$ \cite{BTM:2026}.

Motivated to capture additional structure of the large cardinal embeddings, we extend the free left distributive algebra $\Acal$ to a canonical algebra $\Ccal_1$ and show that $\Ccal_1$ has certain desirable properties. The demonstration of these properties is the second of our two main results.

\begin{theorem*}[B]
    Under appropriate large cardinal assumptions, the structure $\mathcal{C}_1$ is \emph{universal} and \emph{homogeneous}, and the application operation on $\mathcal{C}_1$ is elementary. Furthermore, certain sets of elementary embeddings that do not embed into $\Acal$ do embed into $\Ccal_1$.
\end{theorem*}

See Corollary \ref{cor:universality} for the precise meaning of universal in the statement Theorem B, Corollary \ref{cor:stronghomogeneity} for that of homogeneous, and Theorem \ref{thm:countablerigidsqrfull} for details of which sets of elementary embeddings embed into $\Ccal_1$. The large cardinal assumptions used in Theorems (A) and (B) are precisely stated later in the paper.

Our theorem shows that the map from $\Acal$ to $\Acal$ given by application by an element $a$ in $\Acal$ can be extended to an elementary embedding from the expanded algebra $\Ccal_1$ to $\Ccal_1$. Laver started with a (set-theoretic) elementary embedding and obtained an LDA; we begin with the application operation on LDAs, which are purely algebraic objects, and obtain, using large cardinals, the existence of a purely algebraic elementary embedding of LDAs. 

The structure of the remainder of the paper is as follows. Section \ref{Background} gives relevant background in model theory and large cardinal theory and a brief summary of important results about left distributive algebras. In Section \ref{section:pullbackalgebra} we prove a theorem connecting divisibility conditions in the algebra $\Acal$ to the existence of pullbacks in the algebra $\Acal_{{\{j\}}}$ of embeddings. In Section \ref{ElemEquivofLDAs} we show that the finitely-generated free left distributive algebras are all $\Sigma_1$-elementarily equivalent to one another but not $\Sigma_2$-elementarily equivalent (Theorem (A)). We then define a canonical extension $\Ccal_1$ of $\Acal$ in Section \ref{CanonicalExtensions} and prove that $\Ccal_1$ preserves additional (relative to $\Acal$) properties of the large cardinal embeddings (Theorem (B)). We close with open questions and possible extensions of this work in Section \ref{OpenQsAndExtensions}.

\section{Background}\label{Background}

Here we offer a brief overview of relevant results about left distributive algebras (LDAs) and their connections to large cardinal axioms. For a comprehensive treatment of the subject, especially self-distributivity and braids, see \cite{Dehornoy:2000}. For a concise survey, see \cite{LaverMiller:2013}.

\subsection{Large Cardinals}\label{LCandLDABackground} 

As mentioned in the introduction, large cardinals are strong logical axioms in the form of cardinals that, if they exist, imply the consistency of ZFC. The large cardinal axioms are (as far as anyone has been able to prove) linearly ordered by consistency strength, and the largest of the large cardinals can be formulated to assert the existence of nontrivial elementary embeddings $j$ from the universe of sets $V$ to a transitive model of set theory $M$ that agrees with $V$ on a large collection of sets. Such elementary embeddings preserve much of the structure of the domain, $V$, and the closer the target model $M$ is to $V$ (or, more generally, the domain), the stronger the large cardinal assertion. Let us now be precise about the definition of an elementary embedding.

Recall that a language $\mathcal{L}$ may include symbols for constants, functions, and relations; an $\mathcal{L}$-structure is a set in which interpretations of the constants, functions, and relations have been assigned. An elementary embedding between $\mathcal{L}$-structures preserves \emph{all} first-order formulas in the language. The language of set theory $\{\in\}$ contains a single binary relation symbol $\in$ interpreted as set membership. An elementary embedding $j$ from $\langle V, \in \rangle$ to a transitive substructure $\langle M, \in \rangle$ preserves every first-order formula $\varphi(v_1, \ldots, v_n)$ in the language $\{ \in \}$ of set theory.

\begin{dfn} Say $j$ is an elementary embedding from $\langle V, \in \rangle$ to a transitive substructure $\langle M, \in \rangle$ if  and only if, for every first-order formula $\varphi(v_1, \ldots, v_n)$ and every tuple $a_1, \ldots, a_n$ of elements of $V$, $$V \models \varphi[a_1, \ldots, a_n] \Leftrightarrow M \models \varphi[j(a_1), \ldots, j(a_n)].$$
\end{dfn}

A nontrivial elementary embedding is one that is not the identity and thus must move some ordinal $\alpha$. We will consider only nontrivial elementary embeddings. For a nontrivial elementary embedding $j$, call the least cardinal $\kappa$ moved by $j$ the \emph{critical point} of $j$, denoted by $\crit(j)$. Letting $\kappa=\crit(j)$, observe that, by elementarity, $\kappa$ must be less than its image: $\kappa < j(\kappa)$. Repeatedly applying $j$ to the image of $\kappa$ under $j$ we get the \emph{critical sequence of $j$}, $\overrightarrow{\crit}(j)$: $$\kappa = \kappa_0, j(\kappa_0) = \kappa_1, \ldots, j(\kappa_n) = \kappa_{n+1}, \ldots.$$

Recalling the right power notation from the previous section, note that $\kappa_n = \crit j^{(n)}$. To make reference to the collection of all critical points of embeddings in $\Acal_j$, we define $$\crit \Acal_j = \{\kappa: \exists k \in \Acal_j  (\crit(k) = \kappa) \}.$$

It is natural to wonder how similar the target model $M$ can be to the whole universe of sets $V$. Maximal similarity would be for $M$ to be $V$ itself: Say $\kappa$ is a \emph{Reinhardt cardinal} if it is the critical point of a nontrivial elementary embedding $j: V \rightarrow V$.\footnote{The definition of a Reinhardt cardinal takes place in von Neumann--G\"odel--Bernays class theory, which allows quantification over classes, rather than in ZFC.} As we will see now, Reinhardt cardinals are not consistent with the Axiom of Choice.\footnote{It is not known whether Reinhardt cardinals under ZF + $\lnot \text{AC}$ are inconsistent, and there are many axioms that trade the Axiom of Choice for closure beyond the limit of compatibility with AC given by Kunen's Theorem. For more on this see, for example, \cite{Goldberg:2024,goldberg2022choiceless}.}

\begin{thm}[Kunen]\label{Kunen'sTheorem}
    Assuming ZFC is consistent, there are no Reinhardt cardinals. Indeed, assuming the Axiom of Choice, for no ordinal $\lambda$ does there exist a nontrivial elementary embedding $j:V_{\lambda+2} \rightarrow V_{\lambda+2}$.
\end{thm}

See \cite{Kamimori:1994} for several proofs of Kunen's Theorem. The first statement of Theorem \ref{Kunen'sTheorem} is the original form of Kunen's Theorem; the second is a corollary. A remarkable recent result of Schlutzenberg, \cite{schlutzenberg2025consistency}, shows that the theory $ZF + $``there is a non-trivial embedding $j: V_{\lambda+2}\rightarrow V_{\lambda+2}$'' is consistent relative to a very strong large cardinal hypothesis called $I_0$.

While there are limits to the closure of the target model if we assume ZFC, some amount of closure of an elementary embedding is guaranteed: For every nontrivial elementary embedding $j: V \rightarrow M$ with $M$ transitive and $\crit(j) = \kappa$, $M$ necessarily contains $V_{\kappa+1}$ as a subset. (See, for example, \cite{BrookeTaylorCramerMiller2024} for a proof of this fact.) 

A cardinal $\kappa$ that is the critical point of a nontrivial elementary embedding $j:V \rightarrow M$ is called a \emph{measurable cardinal}. If the image $M$ contains all of its own $\lambda$-sequences for some $\lambda > \kappa$, it is called \emph{$\lambda$-supercompact}. If it is closed under subsets of size $j(\kappa)$, it is called \emph{huge}. In the stated order, these axioms are progressively stronger in consistency strength, and there are many others between and beyond them. The axioms of interest to us here are those for which the domain and target models are both rank initial segments of the universe $V_{\delta}$ for some ordinal $\delta$. Such cardinals are larger than huge cardinals.

For $\lambda$ a limit ordinal of cofinality $\omega$, a nontrivial elementary embedding $j: V_{\lambda} \rightarrow V_{\lambda}$ is a \emph{rank-to-rank} embedding, also called an $I_3$ embedding. We frequently talk about the collection of all such embeddings, so we reiterate the following definition.

\begin{dfn}
    Suppose $\lambda$ is a limit ordinal of cofinality $\omega$. Define $\Ecal_\lambda$ to be the set of nontrivial elementary embeddings $j: V_{\lambda} \to V_{\lambda}$.
\end{dfn}

There are a number of ways to incrementally strengthen the rank-to-rank hypothesis without exceeding the known threshold at $\lambda+2$ for inconsistency with the Axiom of Choice. In essence, each level of the strengthening preserves progressively more complex formulas. We give the definitions of these levels below and will use some of them extensively in the sequel, but before stating them, we offer a brief review of the Levy hierarchy of formulas for first-order languages $\mathcal{L}$ and the language of set theory $\mathcal{L}_{\in}$.

For a first-order language $\mathcal{L}$, a quantifier-free formula is considered both $\Sigma_0$ and $\Pi_0$. If the formula $\psi$ is $\Sigma_n$, then the formula $\forall v_1 \ldots \forall v_k \psi$ is $\Pi_{n+1}$, and, likewise, if $\psi$ is $\Pi_n$, then $\exists v_1 \ldots \exists v_k \psi$ is $\Sigma_{n+1}$. 

For $\mathcal{L}_{\in}$, an existential formula $\varphi$ in which all quantifiers are bounded ($\exists v \in w$) is called $\Sigma_0$. Similarly, a universal formula $\varphi$ with only bounded quantifiers ($\forall v \in w$) is $\Pi_0$. If the formula $\psi$ is $\Sigma_n$, then the formula $\forall v_1 \ldots \forall v_k \psi$ is $\Pi_{n+1}$, and, likewise, if $\psi$ is $\Pi_n$, then $\exists v_1 \ldots \exists v_k \psi$ is $\Sigma_{n+1}$. 

\begin{dfn}
Say that an elementary embedding $j$ from $\mathcal{M}$ to $\mathcal{N}$ is $\Sigma_m$-elementary if for all $\Sigma_m$-formulas $\varphi(v_1, \dots, v_k)$ and all $x_1, \dots, x_k$ in $\mathcal{M}$, $$\mathcal{M} \models \varphi[x_1, \ldots , x_k]  \text{ if and only if } \mathcal{N} \models \varphi[j(x_1), \ldots, j(x_k)].$$
\end{dfn}

Two $\mathcal{L}$-structures $A$ and $B$ are \emph{$n$-equivalent}, denoted by $A \equiv_n B$ and also called \emph{$\Sigma_n$-equivalent}, if every $\Sigma_n$-sentence (equivalently, every $\Pi_n$-sentence) is true in $A$ if and only if it is true in $B$. Say $A$ and $B$ are \emph{elementarily equivalent}, denoted by $A \equiv B$, if every (first-order) sentence is true in $A$ if and only if it is true in $B$. 

In the case that $\mathcal{M}$ and $\mathcal{N}$ are sets, we similarly define $\Sigma^1_n$-elementary, using second-order quantifiers (that is, quantifying over subsets). 

It is now possible to define the strengthenings of I3 mentioned earlier. The embeddings here are given in order of increasing consistency strength.

\begin{description}
\item[Axiom I3:] There exists a nontrivial elementary embedding $j: V_\lambda \to V_\lambda$.
\item[Axiom I2:] There exists a nontrivial elementary embedding $j: V \to M$ for some transitive class $M$ such that the supremum of the critical sequence of $j$ is $\lambda$ and $V_{\lambda} \subseteq M$.
\item[$\Sigma_n^1$-elementary rank-to-rank] There exists a nontrivial $\Sigma_n^1$-elementary embedding $j: V_\lambda \to V_\lambda$.\footnote{Note that this equivalent to $j$ extending to a $\Sigma_n$-elementary embedding $\vlm \to \vlm$.}
\item[Axiom I1:] There exists a nontrivial (fully) elementary embedding $j: \vlm \to \vlm$. 
\item[Axiom I0:] There exists a nontrivial elementary embedding $$j: L(\vlm) \to L(\vlm).\footnote{$L(\vlm)$ is G\"{o}del's constructible universe built on top of $\vlm$. }$$
\end{description}

We now turn to the application operation.

\begin{dfn}[Application]
Let $j: V_{\lambda} \rightarrow V_{\lambda}$ be a rank-to-rank embedding. For $A \subseteq V_{\lambda}$, define $$j \cdot A = jA = j(A) = \bigcup_{\alpha < \lambda} j(A \cap V_{\alpha}).$$
\end{dfn}

Note that because $j$ itself is a subset of $V_{\lambda}$, the application operation so defined allows for the application of an embedding $j$ in $\Ecal_{\lambda}$ to itself. Furthermore, by elementarity, the application operation is left distributive.

Rank-to-rank embeddings can be naturally extended to embeddings on subsets of $V_{\lambda}$, namely to $\Sigma_0$-embeddings $j^+:V_{\lambda+1} \rightarrow V_{\lambda+1}$, through the application operation: Any $j: V_\lambda \to V_\lambda$ uniquely determines a function $j^+: \vlm \to \vlm$ by
$$j^+(A) = \bigcup_{\alpha < \lambda} j(A \cap V_\alpha),$$
for $A$ in $\vlm$.
We call $j^+$ \emph{the extension of $j$ to $\vlm$.}

For embeddings $j$ and  $k$ in $\Ecal_\lambda$, we refer to $j^+(k)$ as \emph{application to $k$ by $j$} and write $jk = j \cdot k = j^+(k)$. 

\subsection{Free Left Distributive Algebras}\label{sec:LDA}

For any cardinal $\kappa$, one can form the free left distributive term algebra $\Acal_{\kappa}$ on $\kappa$ generators and one binary operation $\cdot$ by forming all terms $A_{\kappa}$ in the generators and $\cdot$ and considering the algebra $\Acal_{\kappa}$ of equivalence classes of those terms under the left distributive law (LD). Two terms $u$ and $v$ in $A_{\kappa}$ are LD-equivalent in $\Acal_{\kappa}$ ($u \equiv_{LD} v$) if and only if one can be obtained from the other by a series of applications of the left distributive law: namely, $u$ can be obtained from $v$ by a series of substitutions of the form $a \cdot (b \cdot c) \leftrightarrow (a \cdot b) \cdot (a \cdot c)$. 

In the sequel we will follow convention and write $\mathcal{E}_{\lambda}$ for the collection of all nontrivial elementary embeddings from $V_{\lambda}$ to $V_{\lambda}$, where $\lambda$ is a limit cardinal of cofinality $\omega$, and also $ab$ for $a \cdot b$ and $a_1a_2a_3 \cdots a_{n-1}a_n$ for $((((a_1 \cdot a_2) \cdot a_3) \cdots a_{n-1}) \cdot a_n)$. Likewise we will write $\Acal = \Acal_1$ for the free left distributive algebra on a single generator,  and, for elementary embeddings $j$ and $k$ in $\mathcal{E}_{\lambda}$,  we will write $\Acal_j$ for $\Acal_{\{j\}}$  and $\Acal_{j,k}$ for $\Acal_{\{j,k\}}$ (and refrain from using $j$ and $k$ as natural numbers in this context).

Laver gave the first nontrivial representation of a free left distributive algebra \cite{Laver:1992}. Specifically, he showed that the algebra of embeddings generated by closing a single nontrivial rank-to-rank elementary embedding $j$ under the application operation (see subsection \ref{LCandLDABackground} for a definition) generates an algebra $\Acal_j$ isomorphic to the free left distributive algebra $\Acal$ \cite{Laver:1992}. To do so, he proved that the iterated left subterm relation $<_L$ is a linear order of $\Acal$ (where, for elements $a$ and $b$ of $\Acal$, $a <_L b$ means that there exist $c_1, \ldots c_n$ in $\Acal$ such that $b = ac_1 \cdots c_n$). He furthermore proved that $\Acal_j \cong \Acal$ and that the word problem for $\Acal$ is solvable: For every pair of words $u$ and $v$ in $\Acal$, exactly one of the following holds: $u <_L v$, $v <_L u$, or $u \equiv_{LD} v$. Laver's result also gives that $\Acal$ is left cancelative: $ab <_L ac$ if and only if $b <_L c$ and $ab = ac$ if and only if $b = c$. 

Laver obtained irreflexivity of $<_L$ from the large cardinal assumption, using the fact that there do not exist rank-to-rank embeddings $k, k_1, k_2, \ldots k_n$ in $\Ecal_{\lambda}$ such that $k = kk_1 \cdots k_n$. Because the algebra of embeddings (and in particular $\Acal_j$) is irreflexive, so too must be the free left distributive algebra. For connectedness ($a \leq_L b$ or $b \leq_L a$ for all $a$ and $b$ in $\Acal_j$), he proved the existence of a normal form and a lexicographic ordering $<_{Lex}$ on normal form terms that agrees with the left subterm relation $<_L$ and the subterm relation $<$. That normal form takes place in a conservative extension $\mathcal{P}$ of $\Acal$: For elements $u$ and $v$ of $\Acal$ and $R$ a relation in $\{<_L, \equiv_{LD}, <_{Lex}, <\}$, we have that $uRv$ holds of $u$ and $v$ as elements of $\Acal$ if and only if $uRv$ holds of $u$ and $v$ as elements of $\mathcal{P}$.

Intuitively, the algebra $\mathcal{P}$ is formed by freely adding a composition-like operation $\circ$ to the application operation, so that $\mathcal{P}$ satisfies the identities in $\Sigma$ for all $a, b$ and $c$ in $\mathcal{P}$: $$\Sigma = \{a \circ (b \circ c) = (a \circ b) \circ c,\  (a \circ b)c = a(bc), \  a(b \circ c) = ab \circ ac, \  a \circ b = ab \circ a\}.$$

The first two equations express that $\circ$ is associative and behaves like composition when interacting with the application operation; the third asserts that left multiplication is still a homomorphism of the algebra; and the final law is now sometimes called the ``braid law'' in reference to the connection between left distributive algebras and Artin's braid group.\footnote{In the braid groups, nonconsecutive generators $\sigma_i$ and $\sigma_j$ commute and consecutive ones satisfy $\sigma_i\sigma_{i+1}\sigma_i = \sigma_{i+1}\sigma_i\sigma_{i+1}$.}

For ease of reference, we summarize here several crucial results of Laver.

\begin{theorem}[Laver \cite{Laver:1992}]\label{Laver'sTheorem}
The algebra $\mathcal{P}$ is a conservative extension of $\Acal$: namely, if two terms in the language of $\Acal$ can be proved equal using $\Sigma$, then they can be proved using just the left-distributive law. (See also \cite{LaverMiller:2013}.)
\end{theorem}

\begin{theorem}\label{Laver'slinearitytheorem} Suppose there exists a nontrivial elementary embedding from $V_{\lambda} \rightarrow V_{\lambda}$.
\begin{enumerate}[(a)]
\item There exists a normal form for terms in $\Acal$, and $<_{Lex}$, the lexicographic ordering on normal form words is a linear ordering of $\Acal$.
\item The lexicographic ordering $<_{Lex}$, the iterated left subterm relation $<_L$, and the subterm relation $<$ all agree\footnote{In this paper we will actually use the \emph{division form theorem} \ref{thm:DFthm}. The existence of the division form was first proven by Laver \cite{Laver:1990} as a consequence of the existence of the normal form; the result that all words in $\Acal$ have a division form equivalent was later proven without relying on the normal form \cite[Theorem 28]{Laver:1992}, \cite{Miller:2007, LaverMiller:2013}.}
\end{enumerate}
\end{theorem}

\begin{corollary}
\begin{enumerate}[(a)]
   \item The algebra $\Acal_j$ formed by closing a single rank-to-rank embedding $j$ under the binary application operation is isomorphic to the free left distributive algebra, $\Acal$;
    \item The iterated left subterm relation $<_L$ linearly orders $\Acal$;
    \item The word problem for $\Acal$ is solvable.
\end{enumerate}
\end{corollary}

Laver later discovered another normal form called the division form that is more useful for applications \cite{Laver:1990}. It is used in the sequel, so we describe it here. The idea of the division form goes approximately like this. Suppose we have words $a$ and $b$ in $\Acal$, where $a$ is an iterated left subterm of $b$ ($a <_L b$). Then find $c_1$, the $<_L$-greatest member of $\Acal$ such that $ac_1 \leq b$. If we do not have equality, find the $<_L$-greatest word $c_2$ such that both $ac_1c_2 \leq_L b$ and $c_2 \leq_L a$. Continue in this manner until we have $b = ac_1 \ldots c_n$, where, for each $1 \leq i \leq n$, $c_i \leq_L c_0c_1 \ldots c_{i-2}$ (letting $a = c_0$). The almost-descending condition is called \emph{normality} and is defined below. The division form theorem \cite{Laver:1990} says that we can find a unique representation of this form for every $u$ in $\Acal$, but the algorithm cannot be carried out in $\Acal$: we must go to $\mathcal{P}$ to state and execute the division form algorithm. This is because $\Acal$ does not contain the least upper bounds of all sequences of words in $\Acal$. For example, for $a, b \in \Acal$, the least upper bound of the sequence $a, b, aba, aba(ab), aba(ab)(aba), \ldots$ (called the iterates of $a$ and $b$) is $a \circ b$, which is not itself in $\Acal$.

Consider, for example, the term $xx(xx)$, where $a = c_0 = x$. Then $c_1 = x$, but $c_2 = xx$ does not satisfy $c_2 \leq_L c_0$. Note that if we use $a = xx$, the term $xx(xx)$ does satisfy the normality condition. This means that normality is a condition that depends on $a$, and we shall define $a$-normal below. Furthermore, normality is a property of literal terms in $A$ (respectively, $P$, $A_k$, and $P_k$), not on equivalence classes in $\Acal$. 

Write $a = a_0a_1 \cdots a_{n-1} \ast a_n$ to denote a parsing of the term $a$, where $\ast$ is  either $\cdot$ or $\circ$: In particular, $a = a_0a_1 \cdots a_{n-1} \cdot a_n$ or $a = a_0a_1 \ldots a_{n-1} \circ a_n$.

To state the division form theorem more formally, we need a definition.

\begin{definition}\label{defn:normality}
    The representation of $w = a_0a_1 \cdots a_{n-1} \ast a_n$ in $P$ is said to be \emph{$a_0$-normal} with respect to the binary relation $<_L$ if $a_2 \leq _L a_0$, $a_3 \leq _L a_0a_1$, and $a_i \leq_L a_0a_1 \cdots a_{i-2}$ for all $i$ such that $2 \leq i \leq n$, and if $n \geq 2$ and $\ast = \circ$, then $a_n <_L a_0 a_1 \cdots a_{n-2}$.
\end{definition}

The last condition of normality is to ensure the uniqueness of the representation by preventing applications of the braid law $a \circ b = ab \circ a$. 

\begin{definition}[Division Form]\label{defn:DF}
    Define DF, the set of division form terms, to be the set of hereditarily $x$-normal terms in $A$. Namely, DF $\subseteq P$ is the smallest set such that $x$ is in DF and if $a_1, \ldots, a_n$ are in DF and $w = xa_1 \cdots a_{n-1} \ast a_n$ is $x$-normal, then $w$ is in DF. Write $|w|$ for the division form of $w$ and $|w|^p$ for the $p$-division form of $w$ (defined below).
\end{definition}

\begin{definition}[$p$-Division Form]
    For any word $p$ in $\mathcal{P}$, define $p$-DF $\subseteq P$, the set of $p$-division form terms, to be the smallest set such that:
    \begin{itemize}
        \item If $w \leq_L p$ then $w$ is in $p$-DF if and only if $w$ is in DF, and
        \item If $p <_L w$, $w = |p|a_1a_2 \cdots a_{n-1} \ast a_n$, $w$ is in $p$-DF if and only if $w$ is $p$-normal, $|p|$ is in DF and each $a_i$ is in $p$-DF.
    \end{itemize}
\end{definition}

We can now state the division form theorem.

\begin{thm}[Division Form Theorem (Laver \cite{Laver:1990})]\label{thm:DFthm}
    Every word in $\Acal$ has a unique division form representation.
\end{thm}

The definitions of normality and of the division forms can be extended to the many generator case. (See \cite{Miller:2016} for definitions and related results.)

Through the connections between free LDAs and braid groups, Dehornoy was able to prove the solvability of the word problem for $\Acal$ without the large cardinal hypothesis (at the expense of the normal form). Dehornoy worked in an extension of Artin's braid group on infinitely many generators $B_{\infty}$ and defined a left distributive operation $[\  ]$. He then showed that a subset of $B_{\infty}$ formed by closing a single element of $B_{\infty}$ under the operation $[\  ]$ gives rise to an algebra isomorphic to $\Acal$ \cite{Dehornoy:1994}.

Importantly, he also proved more. In the case of free LDAs on more than one generator, $<_L$ is no longer a linear ordering of $\Acal_n$. Say that $u$ and $v$ in $\Acal_n$ have a \emph{variable clash}, denoted $u \nsim v$, if and only if there exist distinct variables $x$ and $y$ such that $ax \leq_L u$ and $ay \leq_L v$ for a (possibly empty) subterm $a$ of $u$ and $v$.

The confluence algorithm compares terms from $\Acal$ by (repeated) substitution of subterms of the form $u(vw)$ by those of the form $uv(uw)$. Dehornoy proved that the confluence algorithm terminates. In addition to resolving which of $a <_L b$, $b <_La$, and $a \equiv_{LD} b$ hold in the one generator case of $\Acal$, confluence also gives information about the iterated left subterm relation $<_L$ on free LDAs with more than one generator. Specifically, he proved that, for every natural number $n$,  $\Acal_n$ is \emph{confluent}. (See \cite{Dehornoy:1989} and \cite[Theorem 3.14]{Dehornoy:2000}). This means that two terms $a$ and $b$ in $\Acal_n$ are equivalent if and only if there exists a term $c$ such that both $a$ and $b$ can be transformed into $c$ using only forward applications of the left distributive law; namely, by replacing subwords of the form $u(vw)$ with $uv(uw)$. Note that we can define $\mathcal{P}_n$ analogously to $\mathcal{P}$ by freely adding a composition operation $\circ$ subject to the identities in $\Sigma$ above. 

In the sequel, we rely frequently on the following quadrichotomy theorem of Dehornoy, which is the best possible generalization of the linearity of $<_L$ to LDAs with more than one generator.

\begin{theorem}[Dehornoy \cite{Dehornoy:2000}]\label{thm:quadrichotomy}
For every natural number $k$ and every pair $u$, $v$ of elements of $P_k$, exactly one of the following holds:
$$u = v, \quad u <_L v ,\quad v <_L u ,\quad u \nsim v.$$
\end{theorem}

We now record some additional properties of LDAs that we will use in the sequel.

\begin{dfn}\label{leftpowers}
    For any natural number $n$ and any $p$ in $\mathcal{P}$,  define the $n^{th}$ \emph{right power} of $p$, denoted by $p^{(n)}$, recursively on $n$ by $$p^{(0)} = p \text { and } p^{(n+1)} = p^{(n)} p^{(n)}.$$
\end{dfn}
We make extensive use of the fact that, for $1 \leq i \leq n$, $p p^{(n)} = p^{(n+1)} = p^{(i)}p^{(n)}$, which holds for all $n$ by a simple induction. For future reference, we isolate it here as Lemma \ref{lem:pullbackproperties}.

\begin{lem}\label{lem:pullbackproperties}
    For every $p$ in $\mathcal{P}$ and every natural number $n$ and every $i \leq n$, $p p^{(n)} = p^{(i)}p^{(n)} = p^{(n)} p^{(n)} = p^{(n+1)}$.
\end{lem}

The following proposition appears to be folklore. We record it for completeness, following the proof in \cite[Theorem 3.14]{delaTorre}.

\begin{prop}
    \label{prop:commonPower} 
       For all $u$ and $v$ in $\Acal$, there exist natural numbers $n$ and $m$ such that $u^{(n)} = v^{(m)}$.
\end{prop}

\begin{proof}
    Let $x$ be the generator of $\Acal = \Acal_x$. It suffices to show that for every $u$ in $\Acal$, there are some $n$ and $m$ such that $u^{(n)} = x^{(m)}$. 

    We induct on the length of $u$. When the length is 1, $u = x$ and we have $u^{(0)} = u = x = x^{(0)}$. Suppose the length of $u$ is $k+1 > 1$. Then $u = v\cdot w$ where the length of $v$ and $w$ are $\le k$. Thus, there are $n_v, n_w, m_v, m_w\in\omega$ such that $v^{(n_v)} = x^{(m_v)}$ and $w^{(n_w)} = x^{(m_w)}$. Since $(a^{(n)})^2 := a^{(n)}a^{(n)} = a^{(n+1)}$ for any $a\in \Acal$, by taking appropriate iterated squares of one of the equations, we may assume that $m_v = m_w$. We then compute 
    $$ u^{(n_w)} = (vw)^{(n_w)} = v\cdot w^{(n_w)} = v\cdot x^{(m_w)} = v\cdot x^{(m_v)} = v\cdot v^{(n_v)} = v^{(n_v+1)} = x^{(m_v+1)}.$$
    Thus, we have $u^{(n)} = x^{(m)}$ for some $n$ and $m \in \omega$, as desired. 
\end{proof}

\subsection{Square Roots}
\begin{dfn}
    For $k,  j \in \Ecal_\lambda$, we say $k$ is a \emph{square root of $j$} if $kk =j$, and $k$ is an \emph{$n$-root of $j$} if $k^{(n)} = j$. 
\end{dfn}

\begin{lem}\label{lem:sqrrtfacts}
    Suppose that $j$ and $k$ are in $\Ecal_\lambda$ and $k$ is a square root of $j$. Then
  \begin{enumerate}
        \item[(i)] The critical point of $k$ is less than that of $j$: $\crit k < \crit j$, and
        \item[(ii)] For any $a \in \vlm$,  $a \in \rng k$ if and only if $k(a) = j(a)$.
    \end{enumerate}
\end{lem}
\begin{proof}
From the elementarity of $k$ and the fact that $kk = j$, we have $k(\crit k) = \crit (kk) = \crit j$; $k (\crit k) > \crit k$, so the first claim holds. For the second part, let $\bar{a}$ be such that $k(\bar{a}) = a$. Then $j(a) = kk(a) = kk(k(\bar a)) = k(k(\bar a)) =k(a)$. If $k(a) = j(a)$, then $k(a)\in \rng j = \rng (kk)$. So $a\in \rng k$.
\end{proof}

\begin{thm}[Laver {\cite[Theorem 13]{Laver:1995}}]
If $j \in \mathcal{E}_{\lambda}$, then for any $p$ and $q$ in $\mathcal{A}_j$, if $p \neq q$, then $p \rest \crit \Acal_j \neq q \rest \crit \Acal_j$. 
\end{thm}

Noting that for any $n>0$, an $n$-root $k$ of $j$ is not in the algebra generated by $j$, we prove the following lemma.

\begin{lem}\label{lem:nthrootfacts}
    If $n < \omega$ and $j, k \in \Ecal_\lambda$ such that $k^{(n)} = j$. Then $k \rest \Acal_j = j \rest \Acal_j$. 
\end{lem}
\begin{proof}
    This follows immediately from Lemma \ref{lem:sqrrtfacts} by induction on $n$. If $n=0$, then $k=j$, so the conclusion clearly holds. Suppose $n > 0$ and suppose that for any $l\leq n$ and $j,k\in \Ecal_\lambda$ such that $k^{(l)}=j$, we have $k\rest \mathcal{A}_j = j\rest \mathcal{A}_j$. Let $j,k\in \Ecal_\lambda$ be such that $k^{(n+1)}=j$. Note that $k^{(n)}k^{(n)}= k^{(n+1)}$. Apply Lemma \ref{lem:sqrrtfacts} to $k^{(n)}, k^{(n+1)}$, using the fact that $\Acal_j \subseteq \rng k^{(n)}$, we have that $k^{(n)}\rest \Acal_j = k^{(n+1)}\rest \Acal_j = j\rest \Acal_j$. Now we apply the induction hypothesis to $k, k^{(n)}$ to get $k\rest \Acal_j = k^{(n)}\rest \Acal_j$.\footnote{Note that $\Acal_j\subseteq \Acal_{k^{(n)}}$.} Therefore $k\rest \Acal_j = j\rest \Acal_j$ as desired.
\end{proof}

\section{Pulling back in the algebra}
\label{section:pullbackalgebra}

For the main results of this paper, proven in Sections \ref{ElemEquivofLDAs} and \ref{CanonicalExtensions}, we make frequent use of the properties of the range function, including when we can sensically discuss the preimage---called a \emph{pullback}---of one embedding under another. The pullback of one embedding by another is itself an elementary embedding, but that embedding is not \textit{a priori} in $\Acal_{j}$, even when both embeddings are. 
In this section, we provide basic definitions and prove properties of pullbacks.  

We are interested in the conditions under which the pullback of one embedding by another is certain to (exist as an embedding and) be contained in a given algebra, such as $\mathcal{A}_j$. Theorem \ref{thm:pullbackiffdivides} of this section relates pairs of elements to the conditions under which their pullbacks are elements of $\mathcal{A}_j$. We furthermore conjecture that Theorem \ref{thm:pullbackiffdivides} can be extended to finitely generated algebras of embeddings in Conjecture \ref{conj:fingenpullbacks}.

\begin{definition}
 For $j, k\in \mathcal{E}_\lambda$, by the \textit{pullback} of $k$ by $j$, we mean $(j^+)^{-1}(k)$. Recall that $j^+$ denotes the extension of $j$ in $\mathcal{E}_\lambda$ to a function $j^+$ from $V_{\lambda+1}$ to $V_{\lambda+1}$. For brevity, we will write $j^{-1}(k)$ for the pullback of $k$ by $j$.
\end{definition}

\begin{lemma}
    For embeddings $j, k \in \Ecal_\lambda$, if $n < \omega$ and $j$ and $jk$ are both $\Sigma^1_n$-elementary, then $k = j^{-1}(jk) \in \Ecal_\lambda$ (the pullback of $jk$ by $j$) is $\Sigma^1_n$-elementary.
\end{lemma}

\begin{proof}
    The proof is very similar to the proof of Theorem 2.4 of \cite{Laver:1997}.
\end{proof}

\begin{lemma} \label{lem:kjiskkj}
    For embeddings $j$ and $k$ in $\mathcal{E}_\lambda$,
    $$j \in \rng k^+ \Leftrightarrow kj = kkj.$$
\end{lemma}
\begin{remark}
For brevity, we will not write the ``$+$" in the following proofs. 
\end{remark}
\begin{proof}
    By the hypothesis that $j$ is in the range of $k$, there is some $\bar j \in \mathcal{E}_{\lambda}$. Observe that $j = k\bar j$, then $kj = k(k \bar j) = kk (k \bar j) = kkj$, where the second equality follows from left distributivity.
    
    For the other direction, suppose $kj=kkj$. Then $kkj$ is the image of $j$ under $k$ and thus in the range of $k$. Since both $kkj$ and $kk$ are in the range of $k$, the pullback of $kkj$ by $kk$, $(kk)^{-1}(kkj) = j$ (being $\Sigma_0$-definable from these embeddings), is in the range of $k$ by elementarity, as desired. 
\end{proof}

\begin{dfn}
    For elements $a$ and $b$ of $\mathcal{A}$ (respectively, $\mathcal{P}$), we say that \emph{$a$ divides $b$}, denoted by $a | b$, if and only if there exists $c$ in $\mathcal{A}$ (respectively, $\mathcal{P}$) such that $b = ac$.
\end{dfn}

\begin{lem}\label{lem:divisoractionequivalence}
    For $a,b \in \mathcal{P}$, $a$ divides $b$ (in $\mathcal{P}$) if and only if $ab = aab$. Similarly if $a,b \in \mathcal{A}$, $a$ divides $b$ (in $\mathcal{A}$) if and only if $ab = aab$.
\end{lem}

\begin{proof}
    The forward direction follows directly from Lemma \ref{lem:kjiskkj}: If $a$ divides $b$, then $b = ac$ for some $c$ in $\mathcal{P}$. Then $ab = a(ac) = aa(ac) = aab$. % say $b = a\bar b$ for some $\bar b \in \mathcal{A}_1$. Then $ab = a(a \bar b) = (aa) (a \bar b) = aa b$.
    
    For the backward direction, suppose that $ab = aab$. We first prove that this entails $a <_L b$. Equality of $a$ and $b$ is impossible by irreflexivity of elementary embeddings (else $aa = aaa$), so suppose for sake of contradiction that $b <_L a$ holds. Then, there exist $a_1, \ldots, a_n$ in $\mathcal{P}$ such that $a = b a_1 a_2 \cdots a_{n-1} \ast a_n,$ where $ \ast $ is either $\cdot$ or $\circ$. 
    But then $$ab = aab = a(ba_1\cdots a_{n-1}*a_n)b = ((ab)(aa_1)\cdots (aa_{n-1})*(aa_n)) b.$$
    The right hand side is $(ab)(aa_1)\cdots (aa_{n-1})(aa_n) b$ when $* = \cdot$ and $$((ab)(aa_1)\cdots (aa_{n-1})\circ(aa_n)) b = (ab)(aa_1)\cdots (aa_{n-1})(aa_n b)$$ when $* = \circ$. 
    Both possibilities contradict the irreflexivity of $<_L$. 

    Assume therefore that $b >_L a$. Let $b = a b_1 b_2 \cdots b_{n-1} * b_n$ be the $|a|$-division form of $b$. We compute:
    $$ab = a(a b_1 b_2 \cdots b_{n-1} * b_n) = aa(ab_1)(ab_2)\cdots (ab_{n-1})*(ab_n),$$
    $$aab = aa(a b_1 b_2 \cdots b_{n-1} * b_n) = aa(ab_1)(aab_2)\cdots (aab_{n-1})*(aab_n).$$
    We claim that both of these are $aa(ab_1)$-normal, and thus in $|aa(ab_1)|$-division form once the individual components are in $aa(ab_1)$-division form. We need to show that $ab_{i} \le_L aa(ab_1)\cdots (ab_{i-2})$ for $2 < i \leq n$, with strict inequality if $\ast = \circ$ and $i = n$. Because $\mathcal{A}$ is left cancelative, multiplying the inequality $b_{i} \le_L ab_1\cdots b_{i-2}$ on the left by $a$ preserves the inequality, including in the case where $i = n$ and $\ast  = \circ$, so the inequality is strict.
    Similarly, applying $aa$ to these inequalities shows us that the second word is also in $|aa(ab_1)|$-normal. The two words therefore have the same $|aa(ab_1)|$-division forms and thus must be identical. In particular, we have $ab_2 = aab_2$ if $n \ge 2$. 
    
    If $n \ge 2$, then the assumption that $ab_1 \cdots b_{n-1}*b_n$ is the $|a|$-division form of $b$ gives that $b_2 \le_L a$. The same argument showing that $a <_L b$ shows that  $ab_2 <_L aab_2$. Therefore, we have both that $ab_2 = aab_2$ and that $ab_2 <_L aab_2$, contradicting the linearity of $<_L$ (see Theorem \ref{Laver'slinearitytheorem}). 

    The only remaining case is $n = 1$ and $b = ab_1$; this proves that $a | b$, as desired. 

    To see that the corresponding statement holds in $\Acal$, note that if $a$ and $b$ are in $\Acal$ and $a$ divides $b$ in $\mathcal{P}$, then $a$ divides $b$ in $\Acal$. In particular, $a|b$ if and only if there exists a $c$ in $\mathcal{P}$ such that $b = ac$, but because $b$ is in $\mathcal{A}$, $c$ must also be in $\mathcal{A}$ (otherwise, it would contain an essential composition and $b$ would be in $\mathcal{P} \setminus \mathcal{A}$---see \cite[page 2158]{LaverMiller:2013}).
\end{proof}

We are now ready to prove the following theorem. 

\begin{thm} \label{thm:pullbackiffdivides}
    Let $j$ be an elementary embedding in $\Ecal_{\lambda}$. For elements $\ell$ and $p$ in $\Acal_j$, the following are equivalent.
    \begin{enumerate}
        \item The pullback of $p$ by $\ell$, $\ell^{-1}p$, exists as an embedding in $\Ecal_{\lambda}$.
        \item The pullback of $p$ by $\ell$, $\ell^{-1}p$, is an embedding in $\Acal_j$.
        \item The word $\ell$ divides $p$ as an element of the algebra $\Acal_j$.
     \end{enumerate}
\end{thm}

\begin{proof}
    (2) implies (1) since $\Acal_j \subset \Ecal_\lambda$.

    (1) implies (3): Let $q = \ell^{-1}p \in \Ecal_\lambda$, so that $\ell q = p$. Thus, $$ \ell p = \ell(\ell q) = (\ell \ell)(\ell q) = \ell\ell p. $$
    By Lemma \ref{lem:divisoractionequivalence}, we have $\ell$ divides $p$ in $\Acal_j$.

    (3) implies (2) is the definition of divisibility.
\end{proof}

(1) implies (2) of Theorem \ref{thm:pullbackiffdivides} is the $n = 1$ case of the following conjecture.

\begin{conjecture}\label{conj:fingenpullbacks}
    Suppose $k_1, \ldots, k_n$ are elementary embeddings in $\mathcal{E}_{\lambda}$ such that $\mathcal{A}_{k_1, \ldots, k_n}$ is free and that $\ell, p \in \mathcal{A}_{k_1, \ldots, k_n}$. If $\ell^{-1} p$ exists, then $\ell^{-1} p \in \mathcal{A}_{k_1, \ldots, k_n}$. 
\end{conjecture}

\section{Elementary Equivalence and Finitely Generated Free LDAs}\label{ElemEquivofLDAs}

One of the first insights gained from manipulating terms in the one- and many-generated free LDAs is that they are qualitatively different due to the existence of variable clashes in the multivariable cases. In this section, we analyze that experiential difference using elementary equivalence. The main result of this section, Theorem \ref{thm:elem_equiv}, says that all finitely-generated free LDAs are $\Sigma_1$-elementarily equivalent but not $\Sigma_2$-elementarily equivalent. Our proof of $\Sigma_1$-elementary equivalence uses large cardinals in an essential way; whether a ZFC proof exists is an open question. The non-$\Sigma_2$-elementary equivalence, however, is provable in ZF. This non-equivalence at the level of $\Sigma_2$ is in contrast with two famous examples from group theory. The first is Tarski's famous conjecture \cite{TarskiSinaceur}, that all free groups (with more than one generator) are elementarily equivalent. Free abelian groups are $\Sigma_2$-elementarily equivalent but not $\Sigma_3$-elementarily equivalent.

In general, square roots of elementary embeddings have weaker elementarity properties than the embeddings from which they are derived. In order to prove Theorem \ref{thm:elem_equiv}, we must first establish a series of lemmas guaranteeing that we can generate the sequences of square roots of embeddings we will need.

The version of Lemma \ref{sqrrtlem} presented here is a slight extension of the construction implicit in \cite[Lemma~2.2]{Laver:1997}, as proven in \cite{BrookeTaylorCramerMiller2024}.

\begin{lem}\label{sqrrtlem} Suppose $j \in \Ecal_\lambda$. Then the following are equivalent.
\begin{enumerate}
    \item The embedding $j$ is $\Sigma_{1}$-elementary.
    \item For any 
$c_0,\ldots, c_s \in \vlm$ there is $k \in \Ecal_\lambda$ that is a square root of $j$ 
such that $c_0, \ldots, c_s\in\rng k$.
    \item For any 
$a_0,\ldots,a_m,$ $b_0, \ldots, b_n \in \vlm$ there is $k \in \Ecal_\lambda$ 
such that $a_0,\ldots,a_m\in\rng k$ and for all $\ell\leq n$, $k(b_\ell)=j(b_\ell)$.
\end{enumerate}
\end{lem}

\begin{proof}
    For (1) implies (2), see \cite[Lemma 16]{BrookeTaylorCramerMiller2024}. (2) implying (3) follows Lemma \ref{lem:sqrrtfacts}.

    For (3) implies (1), suppose that $$\vlm \models \exists y_0,\ldots, y_m \phi[c_0,\ldots,c_n,y_0,\ldots, y_m],$$ where $\phi$ is $\Sigma_0$ and $c_0, \ldots, c_n$ in the range of $j$ with preimages $b_0, \ldots, b_n$, and suppose $a_0,\ldots, a_m$ witness this statement. Then there exists a $k$ such that $a_0,\ldots, a_m \in \rng k$ and $k(b_i) = j(b_i)$ for $i < n$. Let $\bar a_0, \ldots, \bar a_m$ be the preimages of $a_0, \ldots, a_m$ under $k$. We have that 
    \begin{align*}
        \vlm \models \phi[c_0, \ldots, c_n, a_0, \ldots, a_m] & \Rightarrow  \vlm \models \phi[b_0, \ldots, b_n, \bar a_0, \ldots, \bar a_m] \\
        & \Rightarrow \vlm \models \exists y_0,\ldots, y_m \phi[b_0,\ldots,b_n,y_0,\ldots, y_m]
    \end{align*}
    So we have the desired result. 
\end{proof}

A similar argument to the previous gives the following:
\begin{lemma}\label{lem:reflection}
    For $n \ge 1$, $j: V_{\lambda+1} \to V_{\lambda+1}$ is a $\Sigma_{2n+1}$ elementary map if and only if for every $a_0,\ldots,a_m,$ $b_0, \ldots, b_n \in \vlm$, there is some $\Sigma_{2n-1}$ elementary map $k:V_{\lambda+1}\to V_{\lambda+1}$ which is a square root of $j$ such that $a_0,\ldots,a_m\in\rng k$ and for all $\ell\leq n$, $k(b_\ell)=j(b_\ell)$.
\end{lemma}

The nature of the strength hypotheses in the previous two lemmas is clarified by the following result of Martin.

\begin{lemma}[Martin]\label{Martin}
    If $j: V_{\lambda+1} \to V_{\lambda+1}$ is $\Sigma_{2n-1}$-elementary for $n \ge 1$, then $j$ is $\Sigma_{2n}$-elementary. 
\end{lemma}

Given the previous lemmas, we make the following definition.

\begin{dfn}
    Define $\Ecal_\lambda^{n}$ to be the set of nontrivial $j: V_{\lambda} \to V_{\lambda}$ that can be extended to $\Sigma_{n}$-elementary embeddings $j^+:\vlm \to \vlm$.
\end{dfn}

In light of Martin's Lemma (Lemma \ref{Martin}) and Lemma \ref{lem:reflection}, we will talk only about embeddings in $\Ecal^{2n}_\lambda$.

We can then rephrase Lemmas \ref{sqrrtlem} and \ref{lem:reflection} as the following.
\begin{lemma}\label{lem:sqrrtsniceformat}
    $j \in \Ecal_\lambda^{2n+1}$ if and only if for every $a_0,\ldots,a_m,$ $b_0, \ldots, b_n \in \vlm$, there is some $k \in \Ecal_{\lambda}^{2n}$ a square root of $j$ such that $a_0,\ldots,a_m\in\rng k$ and for all $\ell\leq n$, $k(b_\ell)=j(b_\ell)$.
\end{lemma}

Lemma \ref{lem:sqrrtsniceformat} allows us to produce arbitrarily long finite sequences of iterated square roots of $j$ when $j^+$ is a fully elementary embedding $j^+: \vlm \to \vlm$. Lemma \ref{lem:sqrrtfacts} and well-foundedness, however, prevent an infinite sequence of such square roots.

\begin{lemma}\label{lem:arblongsqrrtseq}
    Suppose $j \in \Ecal_\lambda$ and $j^+: \vlm \to \vlm$ is fully elementary. Then for any $p, q < \omega$ and for any $a_0,\ldots,a_m,$ $b_0, \ldots, b_n \in \vlm$, there is a sequence of embeddings $k_i \in \Ecal_\lambda^{2(i+q)}$  for $i \le p$ such that:
    \begin{enumerate}
        \item $k_p$ is a square root of $j$;
        \item For all $i < p$, $k_i$ is a square root of $k_{i+1}$; 
        \item For all $i \le p$, $a_0,\ldots,a_m$ are in the range of $k_i$ and for all $\ell\leq n$, the $k_i$ agree with $j$ on the $b_{\ell}$: $k_i(b_\ell)=j(b_\ell)$; and
        \item Each of $a_0, \ldots, a_m $ are in the range of $k_p \circ k_{p-1} \circ \cdots \circ k_0$.
    \end{enumerate}
\end{lemma}

\begin{proof}
    We simply apply Lemma \ref{lem:sqrrtsniceformat} $p+1$ many times to $a_0,\ldots, a_m, b_0, \ldots, b_n$ in the natural way to obtain the sequence of embeddings $k_p, k_{p-1}, \ldots, k_1, k_0$, with $k_p$ a square root of $j$, $k_{p-1}$ square root of $k_p$, etc. It is clear from this process that properties (1)-(3) hold. To see that (4) holds, note, for instance, that $a_0 \in \rng k_p$. Since $k_{p-1}$ is a square root of $k_p$, so $k_p\in \rng k_{p-1}$, and $a_0 \in \rng k_{p-1}$, we have that $k_p^{-1}(a_0) \in \rng k_{p-1}$. Hence $a_0$ is in the range of $k_{p} \circ k_{p-1}$. Similarly, by induction, $a_0 \in \rng k_p \circ \cdots \circ k_0$. Thus (4) clearly follows. 
\end{proof}

The proof of Theorem \ref{thm:elem_equiv} uses the following Theorem of Laver-Steel.
\begin{theorem}[Laver-Steel \cite{Laver:1995}]
    Let $j \in \Ecal_\lambda$. Then $\crit \Acal_j$ has order type $\omega$. 
\end{theorem}

We are now ready to prove the theorem. In part (1) of the theorem, the Large Cardinal (LC) assumption is precisely the existence of a nontrivial elementary embedding $V_{\lambda+1} \rightarrow V_{\lambda+1}.$

\begin{theorem}\label{thm:elem_equiv}
    \begin{enumerate}
        \item (LC) Assume that there exists a nontrivial elementary embedding $\vlm \to \vlm$, then  $\Acal_n \equiv_1 \Acal_m$ for all $n,m \le \omega$.
        \item (ZF) $\Acal_n \not \equiv_2 \Acal_m$ for all $n \neq m \le \omega$. 
    \end{enumerate}
\end{theorem}

\begin{proof}
 For (1), we first show $\Acal_2 \equiv_1 \Acal_1$. Since $\Acal_1$ can be naturally embedded into $\Acal_2$, it suffices to show that any $\Sigma_1$-statement true in $\Acal_2$ is also true in $\Acal_1$. 
 
 Let $j \in \Ecal_{\lambda+1}$. Following \cite{BrookeTaylorCramerMiller2024}, we consider the embeddings $k$ and $\ell$ defined by $k = j \circ j^{(1)} \circ j^{(2)} \circ \cdots$ and $\ell =  j \circ j \circ j^{(1)} \circ j^{(2)} \circ \cdots$.  By \cite{BrookeTaylorCramerMiller2024}, $k, \ell \in \Ecal_\lambda$ and $\mathcal{A}_{k,\ell}$ is a free, 2-generated LDA. 
 
 Suppose $\Acal_{k,\ell}\models \exists \bar x\varphi'(\bar x)$ where $\varphi'(\bar x)$ is quantifier-free. Since $\Acal_{k,\ell}$ is generated by $k$ and $\ell$, we can replace the witness of $\varphi'$ by terms in $k$ and $\ell$, so we may assume that $\Acal_{k,\ell}\models \varphi(k,\ell)$ where $\varphi$ is quantifier-free. It suffices to show that $\Acal_1\models \exists k^* \exists \ell^* \, \varphi(k^*,\ell^*)$. We may assume that $$\varphi(x,y) = \bigwedge\limits_i s_i(x,y) \neq t_i(x,y) \wedge \bigwedge\limits_i s'_i(x,y) = t'_i(x,y)$$ where $s_i(x,y), t_i(x,y), s'_i(x,y), t'_i(x,y)$ are terms in the language of LDAs involving $x,y$.
 
 We define for $n < \omega$, $k_n = j \circ j^{(1)} \circ j^{(2)} \circ \cdots \circ j^{(n)}$ and $\ell_n = j \circ j \circ j^{(1)} \circ j^{(2)} \circ \cdots\circ j^{(n)}$. 
 Let $\la  \delta_n|\, n < \omega \ra$ be the increasing enumeration of $\crit \Acal_j$. We have that for any $n < \omega$, $k_n\restriction V_{\delta_n} = k\restriction V_{\delta_n}$, and $\ell_n\restriction V_{\delta_n} = \ell\restriction V_{\delta_n}$, since $\delta_n \le \crit j^{(n)}$. 
 
 Assume $n$ is large enough so that, for all the (finitely many) $i$'s in $\varphi$ $$s_i(k_n,\ell_n)\rest V_{\delta_n} \neq t_i(k_n,\ell_n)\rest V_{\delta_n}.$$ By Theorem 10 of \cite{Laver:1995} there are embeddings $k^*$ and $\ell^*$ in $\mathcal{A}_j$ such that $$k^* \rest V_{\delta_n} = k_n \rest V_{\delta_n}$$ and $$\ell^* \rest V_{\delta_n} = \ell_n \rest V_{\delta_n}.$$ We then have that for all $i$, $$s_i(k^*, \ell^*) \neq t_i(k^*, \ell^*).$$

    Recall that, in an equational class (in the sense of universal algebra), any positive atomic formula true about the generators of a free structure is true about any tuple in any structure in the class \cite[Theorem 11.4]{BurrisSankappanavar}. In our setting, since $k$ and $\ell$ are free generators of $\Acal_{k,\ell}\cong A_2$, any two elements of any LDA will satisfy any equations that $k$ and $\ell$ satisfy. In particular, since $s_i'(k,\ell) = t_i'(k,\ell)$, we have that $s_i'(k^*,\ell^*) = t'_i(k^*,\ell^*)$. So $k^*$ and $\ell^*$ satisfy $\varphi$, namely, $$ \Acal_1 \models \varphi(k^*,\ell^*) $$
    as desired. 
    
    We now consider the other case of (1) and show that $\Acal_n \equiv_1 \Acal_m$ for $n,m \ge 2$. It suffices to prove that $\Acal_3$ embeds into $\Acal_2$, then proceeds by induction. By Brooke-Taylor-Cramer-Miller-Edwards \cite{BrookeTaylorCramerMiller2024}, we can find $k,\ell$ with disjoint iterated ranges. By Lemma \ref{lem:reflection}, there are some elementary embeddings $j$ and $\bar \ell$ such that $k = \bar\ell j$ and $ \ell = \bar\ell\bar\ell$. Then $\langle j, k, \ell\rangle = \langle j, \bar\ell j, \bar\ell \bar \ell \rangle \subseteq \langle j, \bar\ell\rangle$. We show that $\langle j, k, \ell\rangle$ is free. This follows since $\irng k,\irng \ell \sseq \irng \bar \ell$. So since $\irng \bar \ell$ and $\irng j$ are disjoint (by elementarity of $\bar \ell$ and the fact that $\irng k$ and $\irng \ell$ are disjoint), this implies that $\irng k$ and $\irng \ell$ are disjoint from $\irng j$. 

    Since $\irng j$ and $\irng \bar \ell$ are disjoint, $\la j, \bar \ell \ra$ is also free. Hence we have that since $\la j, k, \ell \ra$ embeds into $\la j, \bar \ell \ra$, that $\Acal_3$ embeds into $\Acal_2$. 
    
    To show (2), notice first that the set of standard generators is definable via the universal formula: $\psi(x) = \forall y, z, x \neq y\cdot z$. Thus, saying the rank of $\Acal_n$ is at least $n$ is $\Sigma_2$ by the sentence $\varphi_n \equiv \exists x_1, \dots, x_n \left(\bigwedge_{i\neq j}x_i \neq x_j \wedge \bigwedge_{1\leq i \leq n} \psi(x_i)\right)$. So if $n > m$, then $\Acal_n \models \varphi_n$ but $\Acal_m \models \neg \varphi_n$. This shows $\Acal_n \not \equiv_2 A_m$.
\end{proof}

\section{Capturing the Algebraic Structure of Elementary Embeddings}\label{CanonicalExtensions}

One of our primary motivating questions is whether algebras of rank-into-rank elementary embeddings always embed into simply definable left distributive algebras. Laver answered this question in the affirmative for algebras with a single generator. %from one elementary embedding. 

We pose this general motivating question as follows.

\begin{question}
\label{qtn:simplydefinablealgembeds}
    Is there a simply definable algebra which embeds $\la E \ra$ for any countable $E \subseteq \Ecal_\lambda$? 
\end{question}

A natural next step beyond the Laver-Steel Theorem would be to consider two rank-to-rank elementary embeddings; namely, given rank-to-rank embeddings $k$ and $\ell$, does $\Acal_{k, \ell} = \la k, \ell \ra$, the algebra generated by closing $k$ and $\ell$ under application and forming equivalence classes under the left distributive law, embed into a simplify definable algebra? 

With certain additional assumptions on our embeddings, we can answer this question in the affirmative. As a warm-up, we describe the idea for proving the following Proposition which is a special case of Corollary \ref{cor:finrigidwksqrfullembeds} below. 
\begin{prop}
    Suppose that $k_0,\ldots, k_n$ are square roots of $j \in \Ecal_\lambda$ such that for all $m < i \le n$, $k_i \in \rng k_m$. Then $\la k_0, \ldots, k_n \ra$ embeds into $\Acal_1$. 
\end{prop}
\begin{proof}[Proof sketch.]
Suppose first that $k$ and $\ell$ are both square roots of an embedding $j$ and that $k$ is in the range of $\ell$. We represent this in the diagram below, where solid lines indicate that the embedding lower in the diagram is a square root of the embedding above it, and dashed lines mean that the lower embedding has the higher embedding in its range.

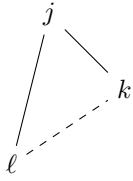
\begin{center}
\begin{tikzpicture}
    \draw (0,0) node (j) {$j$};
    \draw (1,-1) node (k) {$k$};
    \draw (-0.5,-2) node (l) {$\ell$};

    \draw (j) -- (k);
    \draw (j) -- (l);
    \draw[dashed] (l) -- (k);
\end{tikzpicture} 
\end{center}

For $k$ and $\ell$ as above, we can show that $\la k, \ell \ra$ embeds into $\Acal$. This is because $\bar k = \ell^{-1}(k)$ generates both $k$ and $\ell$: $\bar k$ is a square root of $\ell$ since we can pull back the fact that $k$ is a square root of $j$ by $\ell$. We can show this in a diagram as follows:

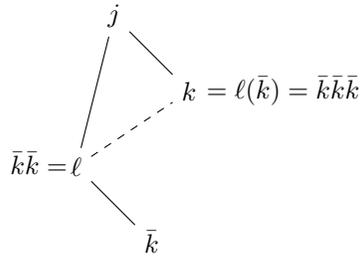
\begin{center}
\begin{tikzpicture}
    \draw (0,0) node (j) {$j$};
    \draw (1,-1) node (k) {$k$};
    \draw (2.25,-1) node (kequiv) {$= \ell(\bar k) = \bar k \bar k \bar k$};
    \draw (-0.5,-2) node (l) {$\ell$};
    \draw (-1,-1.95) node (lequiv) {$\bar k \bar k = $};
    \draw (0.5,-3) node (kbar) {$\bar k$};

    \draw (j) -- (k);
    \draw (j) -- (l);
    \draw[dashed] (l) -- (k);
    \draw (kbar) -- (l);
\end{tikzpicture} 
\end{center}

But $\ell(\bar k) = k$, so we have that $\la k, \ell \ra$ is a subset of $\Acal_{\bar k} \isom \Acal$. 

This argument generalizes to the general case. We illustrate this for $n=2$ in the following diagram:

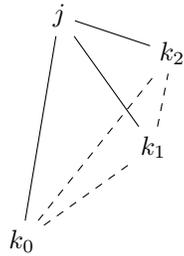
\begin{center}
\begin{tikzpicture}
    \draw (0,0) node (j) {$j$};
    \draw (1.5,-0.5) node (k3) {$k_2$};
    \draw (1.25,-1.75) node (k2) {$k_1$};
    \draw (-0.5,-3) node (k1) {$k_0$};

    \draw (j) -- (k3);
    \draw (j) -- (k2);
    \draw (j) -- (k1);

    \draw[dashed] (k2) -- (k3);
    \draw[dashed] (k1) -- (k2);
    \draw[dashed] (k1) -- (k3);

\end{tikzpicture} 
\end{center}

Pulling back $k_2$ by $k_1$ we get the following:

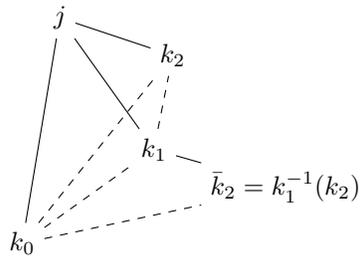
\begin{center}
\begin{tikzpicture}
    \draw (0,0) node (j) {$j$};
    \draw (1.5,-0.5) node (k3) {$k_2$};
    \draw (1.25,-1.75) node (k2) {$k_1$};
    \draw (-0.5,-3) node (k1) {$k_0$};
    
    \draw (3,-2.25) node (k2bar) {$\bar k_2 = k_1^{-1}(k_2)$};

    \draw (j) -- (k3);
    \draw (j) -- (k2);
    \draw (j) -- (k1);

    \draw[dashed] (k2) -- (k3);
    \draw[dashed] (k1) -- (k2);
    \draw[dashed] (k1) -- (k3);

    \draw (k2bar) -- (k2);
    \draw[dashed] (k1) -- (k2bar);
\end{tikzpicture} 
\end{center}

Note that $\bar k_2 \in \rng k_0$ since $k_1, k_2 \in \rng k_0$. 
Pulling back $k_2$, $k_1$ and $\bar k_2$ by $k_0$ to $k_2^*$, $k_1^*$, $\bar k_2^*$ we obtain the following diagram:

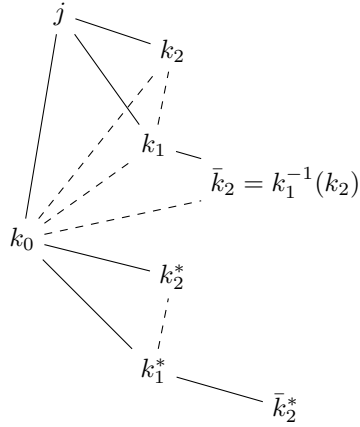
\begin{center}
\begin{tikzpicture}
    \draw (0,0) node (j) {$j$};
    \draw (1.5,-0.5) node (k2) {$k_2$};
    \draw (1.25,-1.75) node (k1) {$k_1$};
    \draw (-0.5,-3) node (k0) {$k_0$};
    
    \draw (3,-2.25) node (k2bar) {$\bar k_2 = k_1^{-1}(k_2)$};

    \begin{scope}[yshift=-3cm]
    \draw (1.5,-0.5) node (k2star) {$k_2^*$};
    \draw (1.25,-1.75) node (k1star) {$k_1^*$};
    
    \draw (3,-2.25) node (k2barstar) {$\bar k_2^*$};
    \end{scope}
    
    \draw (j) -- (k2);
    \draw (j) -- (k1);
    \draw (j) -- (k0);
    % \draw (j) -- (k0);
    \draw[dashed] (k1) -- (k2);
    \draw[dashed] (k0) -- (k1);
    \draw[dashed] (k0) -- (k2);
    
    \draw (k2bar) -- (k1);
    \draw[dashed] (k0) -- (k2bar);

    \draw (k1star) -- (k0);
    \draw (k2star) -- (k0);
    \draw (k2barstar) -- (k1star);
    \draw[dashed] (k1star) -- (k2star);
\end{tikzpicture} 
\end{center}

The reader can check that $k_0, k_1, k_2 \in \Acal_{\bar k_2^*} \isom \Acal$.
\end{proof}

In the next subsection, we describe how this process can be extended to a general class of finite structures. Extending to countably infinite structures requires moving beyond $\Acal_1$ (see Theorem \ref{thm:countablerigidsqrfull}).

\subsection{Rigid and Square-full Sets of Embeddings}

To describe the main theorem of this section, Theorem \ref{thm:countablerigidsqrfull}, we need the following definitions.

\begin{dfn}
    Suppose that $S \subseteq \Ecal_\lambda$ and $j, k \in \Ecal_\lambda$. We define the following:
    \begin{enumerate}
        \item $S$ is \emph{rigid} if it is linearly ordered by the range relation. 
        \item $k$ is an \emph{iterated square root} of $j$ if for some $n < \omega$, $k^{(n)} = j$.
        \item $S$ is \emph{square-confluent} if for any $j,k \in S$, there exists $n, m < \omega$ such that $j^{(n)} = k^{(m)} \in S$.
        \item $S$ is \emph{square-full} if for all $k \in S$ either $kk \in S$ or every $\ell \in S\setminus\{k\}$ is an iterated square root of $k$.
        \item A rigid $S= \{ k_0, k_1, \ldots, k_n\}$ such that for $i < j$, $k_i \in \rng k_j$ is \emph{weakly square-full} if for $1 \le i \le n$, $k_i$ is an iterated square root of $k_0$ and $k_i k_i \in \Acal_\ell$ for some $\ell \in \{ k_0, \ldots, k_{i-1}\}$.
    \end{enumerate}
\end{dfn}

\begin{rmk}
\begin{enumerate}
    \item Note that while any weakly square-full set $S$ can be extended to a square-full set $S'$, it is not that we can always preserve rigidity for such sets. For instance, suppose $j,k  \in \Ecal_\lambda$, and $k$ is a square root of $jjj$ such that $j \in \rng k$. Set $S = \{j(jj), jj, j, k\}$. Since $jjj \notin \rng j$ and $kk = jjj$, we cannot extend $S$ to a rigid, square-full set. 
    \item Note that if $S$ is finite and square-full, then $S$ must be square-confluent as well, as there must be a top element.    We will use this fact largely without mention below.
    \item Note that any subset of a rigid set is also rigid. 
\end{enumerate}
\end{rmk}

For the next lemma, we define the \emph{right depth} $r(t)$ for a term $t$ representing some word in $\Acal=\Acal_x$ recursively by $r(x) = 1$ and $r(t) = r(v)+1$ if $t = uv$. Dehornoy \cite{Dehornoy:2000} showed that this is well-defined on $\Acal$, namely, if $t$ and $t'$ represent the same word in $\Acal$, then $r(t) = r(t')$. 

\begin{lem}\label{lem:purelyalgebraicrootboundfact}
    Let $a$ be in $\mathcal{A}$ with $a \neq x$ and suppose there exist $m, n < \omega$ such that $a$ is an $m^{th}$ root of $x^{(n)}$, namely, $a^{(m)} = x^{(n)}$. Then $m < n$. 
\end{lem}

\begin{proof}
    Note that $r(a^{(m)}) = r(a)+m$, and that $r(a) \ge 2$ if $a \neq x$. Thus, from $a^{(m)} = x^{(n)}$, we have $m+r(a) = n+1$, so we have $m < n$. 
\end{proof}

\begin{lem}\label{lem:boundonrootsweaklysquarefull}
    If $S$ is rigid, weakly square-full, and finite and $S= \{ k_0, k_1, \ldots, k_n\}$ such that for $i < j$, $k_i \in \rng k_j$, then
    \begin{enumerate}
        \item $k_1$ is a square root of $k_0$,
        \item For all $1 \le i \le n$, $k_i$ is an $m$th root of $k_0$ for some $m \le i$. 
    \end{enumerate}
\end{lem}

\begin{proof}
    Note that (1) is a special case of (2), and is the base case of the induction to prove (2). To see (1) by definition of weakly square-full we have $k_1 k_1 \in \Acal_{k_0}$. But if $k_1 k_1$ is not $k_0$ then it must be an iterated square root of $k_0$ which is not possible. So it must be $k_0$. 

    We prove (2) by induction on $i$. We want to show that $k_i k_i$ is an $m$th root of $k_0$ for some $m \le i$. We have $k_i k_i \in \Acal_{k_{p}}$ for some $p < i$. Since $k_{p}$ is an $m'$th root of $k_0$ for some $m' \le p$, using Lemma \ref{lem:purelyalgebraicrootboundfact} on $\Acal_{k_p}$ and considering $k_0 = k_p^{(m')}$ and $k_i k_i$ as an iterated root of $k_0$, $k_i k_i$ is an $m''$th root of $k_0$ for some $m'' \le m' \le p < i$ (the equality holding if $k_p = k_i k_i$). So we are done. 
\end{proof}

\begin{lem}\label{lem:rigidsqrfull}
    Suppose $m < \omega$ and $S \subseteq \Ecal_\lambda^m$ is a set of elementary embeddings such that $S$ is finite, rigid, and weakly square-full. Write $S= \{ k_0, k_1, \ldots, k_n\}$ such that for $i < j$, $k_i \in \rng k_j$.
    
    Then $S \subseteq \mathcal{A}_\ell$ for 
    $$\ell = (k_1 \circ k_2 \circ \cdots \circ k_{n})^{-1}(k_0).$$
    Also $\ell \in \Ecal_\lambda^m$ is an $n$th root of $k_0$ for $n = |S|-1$. 
\end{lem}

    In this case, we say that the $\ell$ is the \emph{natural pullback given by $S$}. By the proof of part (4) in Lemma \ref{lem:arblongsqrrtseq}, the natural pullback given by $S$ always exists for sets $S$ that are finite and rigid.

    \begin{dfn}
        Suppose $S\subseteq \Ecal_\lambda$ is nonempty, finite, and rigid. Write $S= \{ k_0, k_1, \ldots, k_n\}$ such that for $i < j$, $k_i \in \rng k_j$. We define the \emph{natural pullback of $S$} to be 
        $(k_1 \circ k_2 \circ \cdots \circ k_{n})^{-1}(k_0).$
    \end{dfn}

    The proof of this lemma is a natural generalization of the `diagram arguments' presented above. To give the reader intuition about how this situation differs, we provide the diagrams for the case that $k_1 k_1 = k_0$, $k_2 k_2 = k_0$, and $k_3 k_3 = k_1$. The diagram of $S= \{k_0, k_1, k_2, k_3 \}$ then looks as follows:
    
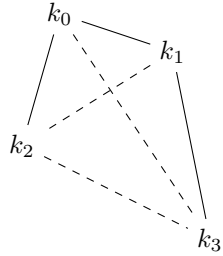
\begin{center}
\begin{tikzpicture}
    \draw (0,0) node (k0) {$k_0$};
    \draw (1.5,-0.5) node (k1) {$k_1$};
    \draw (-0.5,-1.75) node (k2) {$k_2$};
    \draw (2,-3) node (k3) {$k_3$};
    
    \draw (k0) -- (k1);
    \draw (k0) -- (k2);
    \draw (k1) -- (k3);
    \draw[dashed] (k1) -- (k2);
    \draw[dashed] (k2) -- (k3);
    \draw[dashed] (k0) -- (k3);
\end{tikzpicture} 
\end{center}

We step by step compute $(k_1 \circ k_2 \circ k_3)^{-1}(k_0)$ and call the sequence of embeddings $\ell_1$, $\ell_2$, $\ell_3$. Pulling back $k_0$ by $k_1$ gives $\ell_1 = k_1$ itself. Pulling this back by $k_2$ we obtain the following:

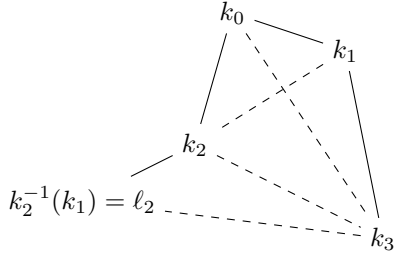
\begin{center}
\begin{tikzpicture}
    \draw (0,0) node (k0) {$k_0$};
    \draw (1.5,-0.5) node (k1) {$k_1$};
    \draw (-0.5,-1.75) node (k2) {$k_2$};
    \draw (2,-3) node (k3) {$k_3$};
    
    \draw (-2,-2.5) node (l2) {$k_2^{-1}(k_1) = \ell_2$};
    
    \draw (k0) -- (k1);
    \draw (k0) -- (k2);
    \draw (k1) -- (k3);
    \draw (l2) -- (k2);
    \draw[dashed] (k1) -- (k2);
    \draw[dashed] (k2) -- (k3);
    \draw[dashed] (k0) -- (k3);
    \draw[dashed] (k3) -- (l2);
\end{tikzpicture} 
\end{center}

Pulling back  $\ell_2$ by $k_3$ we obtain:

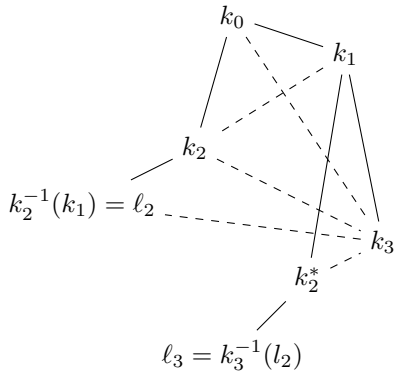
\begin{center}
\begin{tikzpicture}
    \draw (0,0) node (k0) {$k_0$};
    \draw (1.5,-0.5) node (k1) {$k_1$};
    \draw (-0.5,-1.75) node (k2) {$k_2$};
    \draw (2,-3) node (k3) {$k_3$};
    
    \draw (-2,-2.5) node (l2) {$k_2^{-1}(k_1) = \ell_2$};

    \draw (1, -3.5) node (k2star) {$k_2^*$};
    \draw (0, -4.5) node (l3) {$\ell_3=k_3^{-1}(l_2)$};
    \draw (k2star) -- (k1);
    \draw (l3) -- (k2star);
    \draw[dashed] (k2star) -- (k3); 
    
    \draw (k0) -- (k1);
    \draw (k0) -- (k2);
    \draw (k1) -- (k3);
    \draw (l2) -- (k2);
    \draw[dashed] (k1) -- (k2);
    \draw[dashed] (k2) -- (k3);
    \draw[dashed] (k0) -- (k3);
    \draw[dashed] (k3) -- (l2);
\end{tikzpicture} 
\end{center}

Note that $\la \ell_2, k_2, k_0, k_1 \ra \isom \la \ell_3, k_2^*, k_1, k_3 \ra$ because the right embeddings are pullbacks of the left embeddings by $k_3$. This implies the structure of the above diagram. So, retracing our steps, we can check that this entire diagram exists in $\Acal_{\ell_3}$, as desired. 

\begin{proof}[Proof of Lemma \ref{lem:rigidsqrfull}.]
Write $S = \{ k_0, k_1, \ldots, k_n\}$ such that for $i < j$, $k_i \in \rng k_j$.

    By inducting on $i$, we will prove that $k_0, \ldots, k_i$ are generated by $(k_1 \circ k_2 \circ \cdots \circ k_{i})^{-1} (k_0) = \ell_i$ and $\ell_i$ is an $n$th root of $k_0$ for some $n$. For the base case, we need to see that letting $l_1 = k_1^{-1}(k_0)$, then $\{k_0,k_1\}\subseteq \mathcal{A}_{l_1}$. By weak square-fullness of $S$ and the way elements of $S$ are enumerated, it follows from Theorem \ref{lem:boundonrootsweaklysquarefull} that $k_1 k_1 = k_0$, so $k_1^{-1}(k_0) = k_1$; therefore, $k_1 = l_1$ and $\{k_0,k_1\}\subseteq \mathcal{A}_{l_1}$.
    
    Assuming it is true for $i$, we prove it for $i+1$. So $k_0, \ldots, k_i$ are generated by $\ell_i$. We have $\ell_{i+1} = k_{i+1}^{-1}(\ell_i)$. We have that $k_{i+1} k_{i+1} \in \Acal_{k_s}$ for some $s \le i$ by weak square-fullness. Hence by elementarity since $k_s$ is generated by $\ell_i = k_{i+1}\ell_{i+1}$, so is $k_{i+1}k_{i+1}$.  Pulling this fact back by $k_{i+1}$ we have that $k_{i+1}$ is generated by $\ell_{i+1}$. But then $\ell_i$ is generated by $\ell_{i+1}$ since
    $$ \ell_{i} = k_{i+1} \ell_{i+1}.$$
    So $k_0, \ldots, k_{i}$ are also generated by $\ell_{i+1}$. 
    Also, since $\ell_i$ is an $i$th root of $k_0$, we have that $\ell_{i+1} = k_{i+1}^{-1}(\ell_i)$ is an $i$th root of $k_{i+1}^{-1} (k_0)$. 
    On the other hand, $k_{i+1}$ is an iterated square root of $k_0$, so $k_{i+1}^{-1} (k_0)^2 = k_0$ by Lemma \ref{lem:pullbackproperties}. 
    Thus, $$\ell_{i+1}^{(i+1)} = k^{-1}_{i+1}(\ell_i^{(i+1)}) = k^{-1}_{i+1}(k_0^{(1)}) = k_0.$$
    So $\ell_{i+1}$ is an $i+1$th root of $k_0$ as desired and we have proved the induction step.
\end{proof}

We note the following interesting corollary of Lemma \ref{lem:rigidsqrfull}.
\begin{cor}\label{cor:finrigidwksqrfullembeds}
    Suppose that $S \subseteq \Ecal_\lambda$ is finite, rigid, and weakly square-full. Then $\la S \ra$ embeds into $\Acal_1$. 
\end{cor}

\begin{lemma}[The rigidity extension lemma]
    \label{lem:rigidityextensionlemma}
    If $S \subseteq \Ecal_{\lambda}^{2m}$ is a set of elementary embeddings with $|S| < m$ such that $S$ is finite, rigid, and weakly square-full, then for any $k \in S$, there is a set $S' \supseteq S$ with the following properties:
    \begin{enumerate}
        \item The elements of $S'$ have a lower bound on their strength (degree of elementarity): $S' \subseteq \Ecal_\lambda^{2(m-|S|)}$;
        \item The set $S'$ is finite, rigid, and weakly square-full; and
        \item The natural pullback by $S'$ is in $\Ecal^{2(m-|S|)}_\lambda$ and is an $n$th root of $k$ for some $n < |S|$. 
    \end{enumerate}
\end{lemma}

\begin{proof}
    Suppose that $S = \{k_0, k_1, \ldots, k_{|S|-1}\}$ where for all $i < j$, $k_i \in \rng k_j$. 
    The natural pullback $\ell$ given by $S$ is an $(|S|-1)$-th root of $k_0$. Let $k = k_p \in S$. Then $k_p$ is an $m'$th root of $k_0$ for some $m' \le p$ by Lemma \ref{lem:boundonrootsweaklysquarefull}. Let $\ell_{m'-1}, \ell_{m'-2}, \ldots, \ell_0$ be given by the proof of Lemma \ref{lem:arblongsqrrtseq} such that:
    \begin{enumerate}
        \item For all $i < m'$, $\ell_i \in \Ecal^{2(m-|S|+i)}_\lambda$
        \item $\ell_{m'-1}$ is a square root of $k$
        \item $\ell_i$ is a square root of $\ell_{i+1}$ for all $i < m'-1$
        \item For all $i < m'$, $S \subseteq \rng \ell_i$. 
    \end{enumerate}
    It is straightforward to check that $S' = S \cup \{ \ell_0, \ldots, \ell_{m'-1}\}$ has the desired properties. In particular, we want to show that the natural pullback by $S'$, call it $\ell'$, is an $n$th root of $k=k_p$ for some $n <|S|$. To see this, we have $\ell$ (the natural pullback given by $S$) is an $m$th root of $k_0$ for $m <|S|$. We have that 
    $$\ell' = (\ell_{m'-1} \circ \cdots \circ \ell_0)^{-1}(\ell).$$
    Lemma \ref{lem:pullbackproperties} restated in terms of pullbacks implies for $i < m'$ that
    $$\ell_i^{-1} ( k^{(i+1)} ) = k^{(i)}.$$ So we have that 
    $$(\ell_{m'-1} \circ \cdots \circ \ell_0)^{-1}(k_0) = (\ell_{m'-1} \circ \cdots \circ \ell_0)^{-1}(k^{(m')})  = k.$$
    So $\ell'$ is an $m$th root of $k$. Since $m < |S|$, we have what we wanted. 
\end{proof}

The following corollary isolates the property we need below of rigid, weakly square-full sets. 

\begin{cor}
    \label{cor:rigidityextlemcor}
    If $S \subseteq \Ecal_{\lambda}^{2m}$ with $|S| < m$ is a set of elementary embeddings which is finite, rigid, and weakly square-full, then for any $k \in S$, there is $\ell \in \Ecal_\lambda^{2(m-|S|)}$, a $(<|S|)$-iterated root of $k$ such that $S \subseteq \Acal_\ell$.
\end{cor}

\subsection{A Canonical Extension of $\Acal_1$}\label{sec:canonicalextension}

We showed in the last subsection that $\Acal_1$ answers Question \ref{qtn:simplydefinablealgembeds} in the affirmative for finite, rigid, and weakly square-full sets of embeddings $E$. It is natural to ask whether this property can be extended beyond such sets.  Countable, rigid, and weakly square-full sets do not (assuming large cardinals), however, in general embed into $\Acal_1$, as a consequence of Laver's theorem on critical points: the order type of $\{\crit k|\, k \in \Acal_j\}$ is $\omega$. To see this, note that any countable rigid set of square roots of a fixed embedding have distinct critical points which are bounded below $\lambda$ (see \cite{Laver:1997} for this argument in the context of inverse limits). 

We define a natural extension of $\Acal_1$ for which we can embed such sets of embeddings (Theorem \ref{thm:countablerigidsqrfull}). This algebra moreover has many interesting and useful properties which help in analyzing $\Acal_1$. 

\begin{dfn}
Let $\pi:  \Acal_1 \to \Acal_1$ be the map $a \mapsto x a$, where $x$ is the generator of $\Acal_1$. We define $\Ccal_1$ to be the direct limit of $\Acal_1 \to \Acal_1 \to \Acal_1 \to \dots$, where every map is $\pi$. We write $x_n$ to mean the image of the generator of the $n$-th copy of $\Acal_1$ in the direct limit. 
\end{dfn}

\begin{rmk}
    Note that $\Ccal_1$ has many of the same properties as $\Acal_1$. For instance, the $<_L$ order lifts to a linear ordering of $\Ccal_1$. Similarly, for $a, b \in \Ccal_1$, there is $n, m < \omega$ such that $a^{(n)} = b^{(m)}$. 
\end{rmk}

We will show below that we can answer Question \ref{qtn:simplydefinablealgembeds} using the algebra $\Ccal_1$ for a certain class of countable subsets of $\Ecal_\lambda$ (Theorem \ref{thm:countablerigidsqrfull}). 

\begin{theorem}\label{lem:elem}
Assume the large cardinal hypothesis: there exists a nontrivial elementary embedding $\vlm \to \vlm$. 
Then for any $z\in \Ccal_1$, the map $i_z: \Ccal_1 \rightarrow \Ccal_1$ given by $i_z(t) = zt$ is elementary.
\end{theorem}

\begin{remark}
    Theorem \ref{lem:elem} seems to go beyond what is asked in Question \ref{qtn:simplydefinablealgembeds} in the sense that it is showing that (the simply definable algebra) $\Ccal_1$ as a structure has properties similar to $\Ecal_\lambda$ under large cardinal assumptions. Note however, that there are some distinctions between the two situations. For instance, if $\Ecal_{\lambda} \neq \emptyset$, then it is not the case that for all $j \in \Ecal_\lambda$, $j^+\rest \Ecal_\lambda: \Ecal_\lambda \to \Ecal_\lambda$ is elementary. For instance,  if $j$ is such that $\crit j$ is least possible, then $j$ has no square roots. But $jj$ does have a square root (namely $j$).
\end{remark}

In order to prove the theorem above, we need some preliminary lemmas which are variations of our square root existence lemma above.

\begin{lemma}
    \label{lem:simplifiedsquarerootstructurelemma}
    Let $j \in \Ecal_{\lambda+1}$ and $k  \in \Acal_{j}$. Then there is a $q \in \Ecal_\lambda$ such that for some $m < \omega$, $q^{(m)}  = j$, and there is some $\ell \in \Acal_q$ such that $\ell$ is a square root of $k$, and $j \in \rng \ell$. 
\end{lemma}

Though we will use the above version of the lemma, we will prove a stronger, more technical version (Lemma \ref{lem:squarerootstructurelemma}). To state this technical version we need the following definition.

\begin{dfn}
   Define the function $\phi_n$ for $n < \omega$ on words in $A_1=\langle x \rangle$ by induction on word length as follows: $\phi_n(x) = n+2$ and $\phi_n(u v) = \phi_n(u) + \phi_{n+\phi_n(u)}(v)$.
\end{dfn}

\begin{rmk}
    Note the following clearly hold of $\phi$:
    \begin{enumerate}
        \item If $n \le m$ then $\phi_n(w) \le \phi_m(w)$ for any word $w$. 
        \item For any $n < \omega$, if $u$ is a literal subterm of $w$ then $\phi_n (u) \le \phi_n(w)$. 
    \end{enumerate}
\end{rmk}

\begin{lemma}
\label{lem:squarerootstructurelemma}
Fix a term $w$ in $A_1$. 
    For any $r < \omega$ and  $j,p \in \Ecal_{\lambda}^{2r}$  elementary embeddings and $m< \omega$ such that $p^{(m)} = j$ and $r \ge \phi_m(w)$, if $k \in \mathcal{A}_j$ as witnessed by $w$ then there exists a $q \in \Ecal_{\lambda}^{2(r-\phi_m(w))}$ such that for some $n < \phi_m(w)$, $q^{(n)} = p$ and there exists $\ell \in \Acal_{q}$ such that $\ell$ is a square root of $k$ and $p$ is in the range of $\ell$.  (See Figure \ref{sqrrtexists} (dotted lines indicate the lower embedding generates the higher embedding).)
\end{lemma}

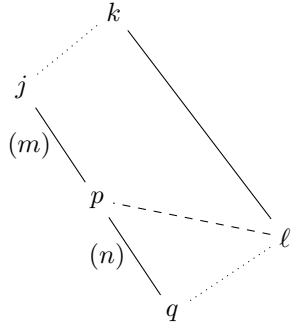
\begin{figure}
\begin{center}
\begin{tikzpicture}
    \draw (1,1) node[anchor=west] (k) {$k$};
    \draw (0,0) node (j) {$j$};
    \draw (1,-1.5) node (p) {$p$};
    \draw (2,-3) node (q) {$q$};
    \draw (3.5, -2) node (l) {$\ell$};

    \draw[dotted] (k) -- (j);
    \draw[dotted] (l) -- (q);
    \draw (k) -- (l);
    \draw[dashed] (l) -- (p);
    \draw (j) -- node[left] {$(m)$} (p);
    
    \draw (p) -- node[left] {$(n)$} (q);
\end{tikzpicture} 
\end{center}
\caption{Embedding configuration for Lemma \ref{lem:squarerootstructurelemma}. The dotted lines indicate that the embedding higher in the diagram is in the range of the embedding lower in the diagram. The $(n)$ and $(m)$ denote an $n$- and $m$-root, respectively.}
\label{sqrrtexists}
\end{figure}

\begin{proof}
    We prove this by induction on the length of $w$. %least length of the word giving $k$ in $\Acal_j$. 
    In the base case $w = x$ and  $k = j$. By Lemma \ref{lem:reflection}, there exists $\ell \in \Ecal_\lambda^{2(r-1)}$ a square root of $k=j$ such that $p \in \rng \ell$. We need to show that there is $q$ such that $q^{(n)} = p$ and $\ell \in \mathcal{A}_{q}$. 

    The set $S = \{j, p^{(m-1)}, p^{(m-2)}, \dots, p, \ell\}$
    is rigid and square-full since $\ell$ is a square root of $j$ and $p \in \rng \ell$. Applying Corollary \ref{cor:rigidityextlemcor} to $S$ and $p$, there is $q \in \Ecal_\lambda^{2(r-1-(m+1))}$, a $<|S|$ iterated-root of $p$ such that $\ell \in \Acal_p$. Note that $|S| = m+1$ and  $r-1-(m+1) \ge 0$ since $r \ge \phi_m(w) = m+2$. So the base case follows.  

    For the general case, we have $k = uv$ and the induction hypothesis holds for words $w_u$ and $w_v$ giving $u$ and $v$ (See Figure \ref{sqrrtexistsindstep}). So $w = w_u w_v$. Applying the induction hypothesis on $j, p,$ and $u$ all in $\Ecal^{2r}_\lambda$, (using that $\phi_m(w) \ge \phi_m(w_u)$), we get $q_u \in \Ecal^{2(r-\phi_m(w_u))}_\lambda$, $\ell_u$, and $n_u < \phi_m(w_u)$ so that $\ell_u \in A_{q_u}$, $p = q_u^{(n_u)}$, and $u = \ell_u^2$. Noting that $j = q_u^{(m+n_u)}$, we again apply the induction hypothesis to $j, q_u, $ and $v$ to get $q_v \in \Ecal^{2(r-\phi_m(w_u) - \phi_{m+n_u}(w_v))}_\lambda, \ell_v,$ and $n_v$ so that $\ell_v \in A_{q_v}$, $q_u = q_v^{(n_v)}$, and $v = \ell_v^2$. Note that
    $$\phi_m(w) = \phi_m(w_u) + \phi_{m+\phi_m(w_u)}(w_v) \ge \phi_m(w_u) + \phi_{m+n_u}(w_v)$$
    since $ n_u \le \phi_m(w_u)$. So in particular, 
    $$2(r-\phi_m(w_u) - \phi_{m+n_u}(w_v)) \ge 0$$
    and $q_v \in \Ecal^{2(r-\phi_m(w))}_\lambda$, as desired. 
    
    We claim that $q = q_v$, $\ell = \ell_u\ell_v$, and $n = n_u + n_v$ satisfies the conclusion of the lemma. First, we compute that $$(\ell_u \ell_v)(\ell_u \ell_v) = \ell_u (\ell_v \ell_v) = \ell_u v = uv = k.$$
    The next to last equality follows from the fact that $v \in \rng \ell_u$, which in turn follows from $p \in \rng \ell_u$ and $v \in A_j = A_{p^{(m)}} \subset A_p$. Then we show that $p \in \rng \ell_u \ell_v$. This follows since $p, \ell_u \in \rng \ell_v$ implies that $\ell_u^{-1}p \in \rng \ell_v$. Applying $\ell_u$ to both sides we get $p \in \rng \ell_u \ell_v$. Lastly, since each of $\ell_u$ and $\ell_v$ is in $A_{q_v}$, $\ell= \ell_u \ell_v$ is also in $A_q = A_{q_v}$.
\begin{figure}
\begin{center}
\begin{tikzpicture}
    \draw (3,3) node (k) {$k= uv$};
    \draw (0,0) node (j) {$j$};
    \draw (0,-2) node (p) {$p$};
    \draw (0,-4) node (qu) {$q_u$};
    \draw (0,-6) node (qv) {$q_v = q$};
    \draw (-2,2) node (u) {$u$};
    \draw (-4,1) node (v) {$v$};
    \draw (-2,-3) node (lu) {$\ell_u$};
    \draw (-4, -5) node (lv) {$\ell_v$};
    \draw (3, -2.5) node (l) {$\ell_u\ell_v = \ell$};

    \draw[dotted] (k) -- (j);
    \draw[dotted] (v) -- (j);
    \draw[dotted] (u) -- (j);
    \draw[dotted] (lu) -- (qu);
    \draw[dotted] (lv) -- (qv);
    \draw[dotted] (l) -- (qv);
    \draw (k) -- (l);
    \draw (u) -- (lu);
    \draw (v) -- (lv);
    \draw[dashed] (l) -- (p);
    \draw[dashed] (lu) -- (p);
    \draw[dashed] (lv) -- (qu);
    \draw (j) -- node[right] {$(m)$} (p);
    \draw (p) -- node[right] {$(n_u)$} (qu);
    \draw (qu) -- node[right] {$(n_v)$} (qv);
\end{tikzpicture} 
\end{center}
\caption{Induction step for Lemma \ref{lem:squarerootstructurelemma}.}
\label{sqrrtexistsindstep}
\end{figure}
\end{proof}

\begin{lemma}
    \label{lem:conesquarerootsexist}
    Assume the large cardinal hypothesis: there exists a nontrivial elementary embedding $\vlm \to \vlm$. 
    Then for all $a, b_0, \ldots, b_m \in \Ccal_1$ there exists an $\bar a \in \Ccal_1$ such that $\bar a \bar a = a$, $\bar a$ left divides $b_i$ and  $\bar a b_i = a b_i$ for all $i \le m$. 
\end{lemma}

\begin{proof}
    Recall that $\Ccal_1$ is the direct limit of the system $\Acal_1 \to \Acal_1 \to \cdots$ where every map is the embedding $\pi$ defined by $x \mapsto xx$. For any $j \in \Ecal_\lambda$, we have $\Acal_1 \isom \Acal_j$. And hence we can similarly identify $\Ccal_1$ with the direct limit of $\Acal_j \to \Acal_j \to \cdots$ where every embedding is defined by $j \mapsto jj$. 

    Fix $j \in \Ecal_{\lambda+1}$ and let $\pi_{n,\omega}: \Acal_j \to \Ccal_1$ be the natural maps, with $\pi_{n,m}: \Acal_j \to \Acal_j$ the factor maps for $n < m$. We have that $\pi_{n,m} = j \circ j \circ \cdots \circ j$, with $(m-n)$-many compositions. Let $a,b_0, \ldots, b_m \in \Ccal_1$. Let $n < \omega$ be such that for some $k_a$ and $k_{b_i}$, $\pi_{n,\omega}(k_a) = a$ and $\pi_{n,\omega}(k_{b_i}) = b_i$ for all $i \le m$.  

    By Lemma \ref{lem:simplifiedsquarerootstructurelemma}, there is a $q\in \Ecal_\lambda$ such that for some $r < \omega$, $q^{(r)} = j$ and there is $\bar k_a \in \Acal_q$ such that $\bar k_a$ is a square root of $k_a$ and $j \in \rng \bar k_a$ (and hence $k_{b_0}, \ldots, k_{b_m} \in \rng \bar k_a$ since these are in $\Acal_j$). 

    Let $\ell = q \circ q \circ \cdots \circ q$ where there are $r$-many compositions of $q$. We have that $\ell(q) = j$, and hence $\ell\rest\Acal_q: \Acal_q \to \Acal_j$ is an isomorphism. 
    
    We claim that $\bar a = \pi_{n+r, \omega}(\ell(\bar k_a))$ satisfies that $\bar a \bar a = a$ and $\bar a$ left divides $b_i$ for all $i \le m$. To see this, note that by Lemma \ref{lem:nthrootfacts} we have $q \rest \Acal_j = j \rest \Acal_j$, and hence 
    $$\ell \rest \Acal_j  = \underbrace{q \circ q \circ \cdots \circ q}_{r\text{ times}} \rest \Acal_j= \underbrace{j \circ j \circ \cdots \circ j}_{r\text{ times}} \rest \Acal_j = \pi_{n,n+r}.$$
    So we have, by elementarity of $\ell$, that since $\bar k_a$ is a square root of $k_a$ that $\ell(\bar k_a)$ is a square root of $\pi_{n,n+r}(k_a) = \ell(k_a)$.\footnote{Note here that we cannot apply $\pi_{n,n+r}$ to $\bar k_a$.} And hence $\bar a = \pi_{n+r, \omega}(\ell(\bar k_a))$ is a square root of 
    $$\pi_{n+r, \omega}(\pi_{n,n+r}(k_a)) = \pi_{n,\omega}(k_a) = a.$$
    
    A similar computations shows that since $k_{b_0}, \ldots, k_{b_m} \in \rng \bar k_a$ that $\bar a$ left divides $b_0, \ldots, b_m$.  This is because for each $i$, $k_{b_i}\in \rng \bar{k}_a$, so by elementarity of $\ell$, $\pi_{n,n+r}(k_{b_i}) = \ell(k_{b_i})\in \rng \ell(\bar{k_a})$. By elementarity of $\pi_{n+r,\omega}$, $b_i = \pi_{n+r,\omega}(\pi_{n,n+r}(k_{b_i}))$ is in the range of $\bar{a} = \pi_{n+r,\omega}(\ell(\bar{k_a}))$.
    By Lemma \ref{lem:divisoractionequivalence} the full result follows. See Figure \ref{limit}.

    \begin{figure}
\begin{center}
    \begin{tikzpicture}
        \draw (-2,0) node (An) {$\Acal_j$};
        \draw (2,0) node (Anplusr) {$\Acal_j$};
        \draw (-2,2) node (Aq) {$\Acal_q$};
        \draw (6,0) node (Cone) {$\Ccal_1$};
        \draw (2,-1) node (lkbara) {$\ell(\bar k_a)$};
        \draw (-2,-2) node (ka) {$k_a$};
        \draw (-2,-1) node (barka) {$\bar k_a$};
        \draw (2,-2) node (pinnplusrka) {$\pi_{n,n+r}(k_a)$};
        \draw (6,-1) node (bara) {$\bar a$};
        \draw (6,-2) node (a) {$a = \pi_{n, \omega}(k_a)$};

        \draw[->] (Aq) -- node[above right] {$\ell\rest \mathcal{A}_q$} (Anplusr);
        \draw[->] (Anplusr) -- node[above] {$\pi_{n+r,\omega}$} (Cone);
        \draw[->] (An) -- node[above] {$\pi_{n,n+r}$}  (Anplusr);

        \draw[|->] (barka) -- node[above] {$\ell$} (lkbara);
        \draw[|->] (ka) -- node[above] {$\pi_{n,n+r} = \ell \rest \Acal_j$} (pinnplusrka);
        \draw[|->] (lkbara) -- node[above] {$\pi_{n+r,\omega}$} (bara);
        \draw[|->] (pinnplusrka) -- node[above] {$\pi_{n+r,\omega}$} (a);
        
    \end{tikzpicture}
    \caption{Diagram of the proof of Lemma \ref{lem:conesquarerootsexist}}
        \label{limit}
    \end{center}
    \end{figure}
\end{proof}

\begin{proof}[Proof of Theorem \ref{lem:elem}]
By induction on complexity of $\phi$, suppose that $z, c_0, \ldots, c_n \in \Ccal_1$ and $\Ccal_1 \models \exists y\, \phi[z \cdot c_0,\ldots, z \cdot c_n, y] $. Let $b \in \Ccal_1$ be a witness to this statement; we want to find a witness in $z\Ccal_1$. We have by Lemma \ref{lem:conesquarerootsexist} that there is $\bar z$ a square root of $z$ such that $\bar z c_i = z c_i$ for all $i \le n$ and $\bar z$ left-divides $b$, say $b = \bar z \bar b$. We then have that by our inductive hypothesis, $$\phi[z c_0,\ldots, z c_n, \bar z \bar b] \implies \phi[\bar{z}c_0, \ldots, \bar z c_i, \bar{z}\bar{b}] \implies \phi[c_0, \ldots c_n, \bar{b}].$$ Pushing forward by $z$ we have that $\Ccal_1 \models \phi[zc_0, \ldots, z c_n,z\bar{b}]$ as desired. 
\end{proof}

\begin{corollary}\label{cor:elemembwitness}
    Assume the large cardinal hypothesis: there exists a nontrivial elementary embedding $\vlm \to \vlm$. Let $z, a_0,\ldots, a_n \in \Ccal_1$. 
    Assume $\phi$ is a formula such that $\exists x \phi[x, a_0, \ldots, a_n]$ holds in $\Ccal_1$. 
    Also, assume that $z$ left-divides each $a_0, \ldots, a_n$. Then there exists some $x^* \in \Ccal_1$ satisfying $\phi[x^*, a_0, \ldots, a_n]$ such that $z$ left-divides $x^*$.
\end{corollary}

\begin{proof}
    This follows by applying the Tarski-Vaught test to the pointwise image of the application embedding given by $z$. More precisely, suppose $a_i = z \bar{a}_i$ for $i\in \{0,\dots, n\}$. By Theorem \ref{lem:elem}, $\Ccal_1\models \exists x \phi[x, \bar{a}_0, \dots, \bar{a}_n]$. Let $x\in\Ccal_1$ witness this, so $\Ccal_1\models \phi[x, \bar{a}_0, \dots, \bar{a}_n]$. Now apply Theorem \ref{lem:elem} to the map $y\mapsto zy$, we have $\Ccal_1\models \phi[zx, z\bar{a}_0, \dots, z\bar{a}_n]$. Since $z\bar{a}_i = a_i$ for $i\in \{0,\dots, n\}$ and $z$ left divides $zx$, we are done.
\end{proof}

\begin{corollary}\label{cor:homogeneity}
Assume the large cardinal hypothesis: there exists a nontrivial elementary embedding $\vlm \to \vlm$.
Suppose $\varphi(x)$ is a first-order formula (with exactly \emph{one} free variable) and that $\Ccal_1 \models \exists x \varphi(x)$ then $\Ccal_1 \models \forall x \varphi(x)$.
\end{corollary}
\begin{proof} This follows from the property true of $\Acal_1$ (and therefore $\Ccal_1$)  that for any $a$ and $b$, there are $n, m < \omega$ such that $a^{(n)} = b^{(m)}$ (Proposition \ref{prop:commonPower}). To see this,
let $x\in \Ccal_1$ be such that $\Ccal_1 \models \varphi(x)$ and let $y$ be any element of $\Ccal_1$. Let $a,b\in A_1$ be the preimages of $x,y$ respectively, as usual. Let $m,n < \omega$ be such that $a^{(m)} = b^{(n)}$. Then we have $x^{(m)} = y^{(n)}$. 
By induction, using the elementarity of the maps $i_z: \Ccal_1 \rightarrow \Ccal_1$ given by $i_z(t) = zt$ for any $z\in \Ccal_1$ as given by Theorem \ref{lem:elem}, we get that $\Ccal_1 \models \varphi(x^{(m)})$, so $\Ccal_1 \models \varphi(y^{(n)})$. This implies, again by Theorem \ref{lem:elem}, that $\Ccal_1 \models \varphi(y)$. 
\end{proof}

Note that the maps $i_z: \Ccal_1 \to \Ccal_1$ for $z \in \Ccal_1$ given by $i_z(t) = zt$ are elementary (by Theorem \ref{lem:elem}) but are never onto. This is because $\rng i_z \subseteq (z, \infty)$, where $(z,\infty) = \{ a \in \Ccal_1|\, z < a\}$. We can however find automorphisms\footnote{An automorphism of $\Ccal_1$ is an embedding that is also a bijection.} of $\Ccal_1$ which map any given element to any other element, which is a stronger version of Corollary \ref{cor:homogeneity}. 

\begin{theorem}[Homogeneity of $\Ccal_1$]
\label{thm:strongerhomog}
Suppose $\Ecal_{\lambda+1} \neq \emptyset$. 
For any $a, b \in \Ccal_1$, there is an automorphism $\pi:\Ccal_1 \to \Ccal_1$ such that $\pi(a) = b$. 
\end{theorem}

This theorem follows from the following lemma.

\begin{lemma}
\label{lem:conesquareseqdecomp}
Suppose $\Ecal_{\lambda+1} \neq \emptyset$. 
    For any $a \in \Ccal_1$ there is a sequence $\la a_n|\, n < \omega \ra$ of elements of $\Ccal_1$ such that
    \begin{enumerate}
        \item $a_0 = a$,
        \item for all $n < \omega$, $a_{n+1}$ is a square root of $a_n$,
        \item $\bigcup_{n < \omega} \Acal_{a_n} = \Ccal_1$. 
    \end{enumerate}
\end{lemma}

\begin{proof}
    We will prove the lemma with the sequence $\la a_n \ra$ satisfying that $a_{n+1}$ is an iterated square root of $a_n$, from which the original statement follows immediately (by simply including the intervening square roots).
    
    Let $j: \vlm \to \vlm$ be an elementary embedding. We realize $\Ccal_1$ as the direct limit of $(\pi_n: A^{(n)}\rightarrow A^{(n+1)} : n<\omega)$ where for each $n$, $A^{(n)} = \la x_n \ra$ and $\pi_n(x_n)= x_{n+1}x_{n+1}$. Let $\pi_{n,\omega}$ be the natural direct limit maps.

    We have that $\bigcup_{n < \omega} \Acal_{\pi_{n,\omega}(x_n)} = \Ccal_1$, and hence given $a \in \Ccal_1$, it is enough to define a sequence $a_n$ of iterated square roots of $a$ such that each $x_n^* = \pi_{n,\omega}(x_n)$ is in $\Acal_{a_m}$ for large enough $m$. We assume without loss of generality that $a$ is an iterated root of $x_0^*$ (using Proposition \ref{prop:commonPower}). Let $t < \omega$ be such that $a^{(t)} = x_0^*$.

\begin{figure}
\begin{center}
\begin{tikzpicture}
    \draw (0,0) node (xmstar) {$x_m^*$};
    \draw (2,1) node (an) {$a_n$};
    \draw (2,3) node (a) {$a=a_0=a_n^{(n')}$};
    \draw (1,5) node (xzero) {$x_0^*$};

    \draw (xmstar) --node[left] {$(m)$} (xzero);
    \draw[dotted] (xmstar) -- (an);
    \draw (an) -- node[right] {$(n')$} (a);
    \draw (a) -- node[above right] {$(t)$} (xzero);

    \draw (5,0) node (j) {$j$};
    \draw (7,1) node (k) {$k$};
    \draw (7,3) node (kn) {$k^{(n')}$};
    \draw (6,5) node (jm) {$j^{(m)}$};
    
    \draw (j) --node[left] {$(m)$} (jm);
    \draw[dotted] (j) -- (k);
    \draw (k) -- node[right] {$(n')$} (kn);
    \draw (kn) -- node[above right] {$(t)$} (jm);

    \draw[-latex] (xmstar) -- node[above] {$\sigma$} (j);

    \draw (7,-1) node (kstar) {$k^*$};
    \draw (7,-2) node (l) {$\ell$};
    \draw (5,-3) node (jstar) {$j^*$};
    \draw (5,-4.5) node (jbar) {$\bar j$};

    \draw (kstar) -- (k);
    \draw (l) -- node[right] {$(<|S|)$} (kstar);
    \draw (jstar) -- (j);
    \draw (jbar) -- node[left] {$(<|S'|)$} (jstar);
    \draw[dotted] (l) -- (j);
    \draw[dashed] (jstar) -- (l);
    \draw[dotted] (jbar) -- (l);

    \draw (2,-2) node (anplusone) {$a_{n+1}$};
    \draw (0,-4.5) node (xustar) {$x_u^*$};

    \draw (xustar) -- node[left] {$(<|S'|+1)$} (xmstar);
    \draw[dotted] (xustar) -- (anplusone);
    \draw (anplusone) -- node[right] {$(<|S|+1)$} (an);
    \draw[dotted] (anplusone) -- (xmstar);

    \draw[-latex] (jbar) -- node[above] {$\pi_{u,\omega} \circ \sigma^*$} (xustar);
    
\end{tikzpicture} 
\end{center}
\caption{Induction step for Lemma \ref{lem:conesquareseqdecomp}.}
\label{iteratedsqrfigure}
\end{figure}
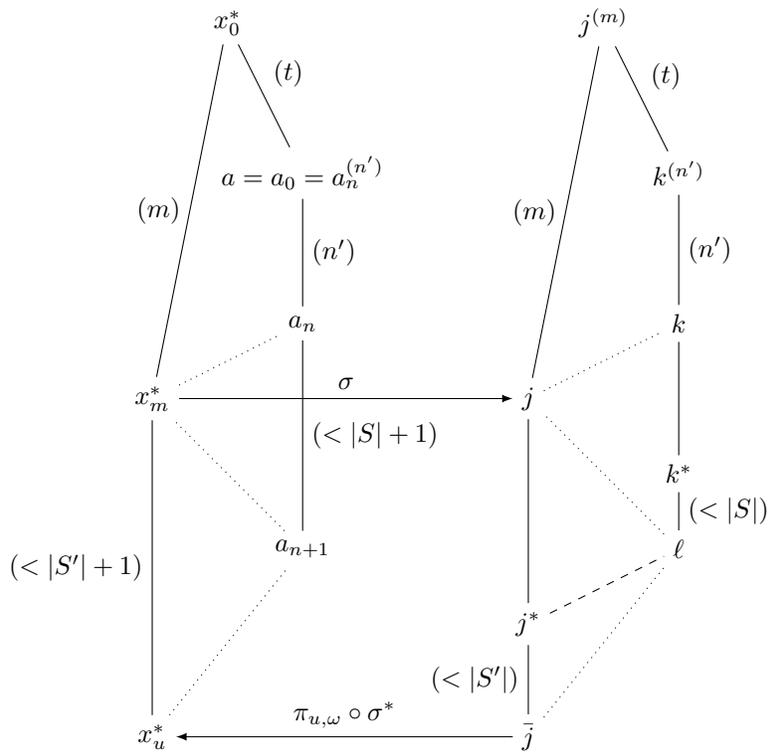

    We define the sequence $\la a_n|\, n <\omega \ra$ by induction such that $x_n^* \in \Acal_{a_{n}}$. Let $a_0 = a$. Suppose $\la a_i|\, i \le n \ra$ is defined such that $x_{n}^* \in \Acal_{a_n}$, and there is $n' < \omega$ such that $a_n^{(n')} = a_0 = a$.  Define $a_{n+1}$ as follows. Assume that $m \ge n+1$ is least such that $a_n \in \Acal_{x_m^*}$. Let $\sigma: \Acal_{x_m^*} \to \Acal_j$ be the natural isomorphism (see Figure \ref{iteratedsqrfigure}). Then $\sigma(x_m^*) = j$ and we let $\sigma(a_n) = k$. Let $S = \{ j^{(m)}, j^{(m-1)}, \ldots, j, k^*\}$ where $k^* \in \Ecal^{2(2(m+2)+t+n'+3)}_\lambda$ is a square root of $k$ (given by Lemma \ref{lem:sqrrtsniceformat}) such that $j \in \rng k^*$. Since $a$ is an iterated root of $x_0^*$ we have that $k$ and hence $k^*$ are iterated roots of $j^{(m)}$. So $S$ is rigid and weakly square-full. By Lemma \ref{lem:rigidityextensionlemma} (since $|S| = m+2 < m+3$) there is an $\ell \in \Ecal_\lambda^{2(m+2+t+n'+3)}$ a $(<|S|)$-iterated root of $k^*$ such that $S \subseteq \Acal_\ell$. Now let $j^*$ be a square root of $j$ such that $\ell \in \rng j^*$. We have that for $s$ such that $\ell^{(s)} = j^{(m)}$ that $S' = \{ \ell^{(s)}, \ell^{(s-1)}, \ldots, \ell, j^*\}$ is a rigid, weakly square-full set. We have 
    $$|S'| = s+2 < (|S| +1+ t+n')+2 = m+3+t+n' + 2,$$ which is clear by looking at the diagram. And hence again by Lemma \ref{lem:rigidityextensionlemma}, there is a $\bar j\in \Ecal_\lambda$ an iterated root of $j^*$ such that $S' \subseteq \Acal_{\bar j}$. 
    
    Let $u < \omega$ be such that $\bar j^{(u)} = j^{(m)}$, and let $\bar a_{n+1}$ be the image of $\ell$ under the natural isomorphism $\sigma^*: \Acal_{\bar j} \isom A^{(u)}$. Then set $a_{n+1} = \pi_{u, \omega}(\bar a_{n+1})$. 
    
    We claim that $a_{n+1}$ has the desired property, namely that $x_{n+1}^* \in \Acal_{a_{n+1}}$. To see this, note that $\pi_{u,\omega}(\sigma^*(\bar j)) = x_u^*$. And since $j \in \Acal_\ell$, $x_m^* \in \Acal_{a_{n+1}}$. So since $m \ge n+1$, the desired result follows.
\end{proof}

\begin{proof}[Proof of Theorem \ref{thm:strongerhomog}.]
    Given $a, b \in \Ccal_1$, let $\la a_n|\, n < \omega \ra$ and $\la b_n|\, n < \omega \ra$ be such that: 
    \begin{enumerate}
        \item $a_0 = a$, $b_0 = b$,
        \item for all $n < \omega$, $a_{n+1}$ is a square root of $a_n$ and $b_{n+1}$ is a square root of $b_n$,
        \item $\bigcup_{n < \omega} \Acal_{a_n} = \bigcup_{n < \omega} \Acal_{b_n} = \Ccal_1$.
    \end{enumerate}
    Then the map sending $a_n \mapsto b_n$ for all $n < \omega$ clearly induces an automorphism of $\Ccal_1$ such that $a \mapsto b$. 
\end{proof}

Recall that a structure $A$ is said to be \emph{strongly $\lambda$-homogeneous} if for every $\bar a, \bar b \in A$ with $|\bar a|, |\bar b| < \lambda$, if $\bar a$ and $\bar b$ have the same type, then there is an automorphism of $A$ sending $\bar a$ to $\bar b$ \cite{Hodges}.

\begin{corollary}
\label{cor:stronghomogeneity}
    Suppose $\Ecal_{\lambda+1} \neq \emptyset$. Then $\Ccal_1$ is strongly $\omega$-homogeneous.
\end{corollary}

\begin{proof}
    Suppose $\vec a = \la a_0, \ldots, a_n \ra$ and $\vec b=\la b_0, \ldots, b_n \ra$ are finite tuples from $\Ccal_1$ which have the same type. Let $a^*\in \Ccal_1$ be such that $\vec a \subseteq \Acal_{a^*}$. Let $w_i$ for $i \le n$ be words such that $w_i[a^*] = a_i$ for all $i \le n$. Since $\vec a$ and $\vec b$ have the same type, there must be a $b^* \in \Ccal_1$ such that $w_i[b^*] = b_i$ for all $i \le n$.

    By Theorem \ref{thm:strongerhomog} there is an automorphism $\pi: \Ccal_1 \to \Ccal_1$ such that $\pi(a^*) = b^*$. But then $\pi$ must map $\vec a$ to $\vec b$, as desired.
\end{proof}

\begin{corollary}[Universality of $\Ccal_1$]\label{cor:universality}
Let $\Ccal_1'$ be the direct limit of the system $(\pi_n':\Acal'^{(n)}\rightarrow \Acal'^{(n+1)} : n<\omega)$ where for each $n$, $\Acal'^{(n)} = \langle x_n' \rangle$ is an isomorphic copy of $\Acal_1$ and $\pi_n'(x_n') = t_n(x_{n+1}')$ for some term $t_n$. Then there is an embedding of $\Ccal_1'$ into $\Ccal_1$.
\end{corollary}

\begin{proof}

We will define an embedding $\sigma: \Ccal_1' \rightarrow \Ccal_1$. 
We define $\sigma(\pi_{n,\omega}'(x_n'))$ by induction on $n$. Set $\sigma(\pi_{0,\omega}'(x_0')) = a$ where $a$ is an arbitrary element of $\Ccal_1$. Suppose now that $\sigma(\pi_{k,\omega}'(x_k'))$ is defined. By Corollary \ref{cor:homogeneity}, there is some $y\in\Ccal_1$ such that $t_k(y) = \sigma(\pi_{k,\omega}'(x_k'))$. Define $\sigma(\pi_{k+1,\omega}'(x_{k+1}')) = y$. 

It is easy to see that $\sigma$ is an embedding.
\end{proof}

\begin{theorem}
    \label{thm:countablerigidsqrfull}
    Assume $\Ecal_{\lambda+1} \neq \emptyset$. 
    Suppose that $S \subseteq \Ecal_\lambda$ is countable, rigid, square-confluent, and square-full. Then $\la S \ra$ embeds into $\Ccal_1$. 
\end{theorem}

Note that, as mentioned at the beginning of Section \ref{sec:canonicalextension}, countable rigid sets of square roots of a fixed embedding do not embed into $\Acal_1$, so Theorem \ref{thm:countablerigidsqrfull} does not hold if we replace $\Ccal_1$ with $\Acal_1$. 

We need the following lemma to prove this theorem.

\begin{lem}  \label{lem:rigidityuniqueness} Suppose that $S,T \subseteq \Ecal_\lambda$ are finite, rigid and square-full and $\phi: S \to T$ is a bijection preserving squares and the range-ordering. Then there is an induced algebra isomorphism $\phi^*: \la S \ra  \to \la T \ra$. 
\end{lem}
The proof of Lemma \ref{lem:rigidityuniqueness} is very similar to the proof of Lemma \ref{lem:rigidsqrfull}. The isomorphism $\phi^*$ is generated by identifying the natural pullbacks of $S$ and $T$, and we must check that this isomorphism agrees with $\phi$. 

\begin{proof} 
Suppose that $k_0, k_1, \ldots , k_n$ is an enumeration of $S$ such that for all $s < t \le n$, $k_s \in \rng k_t$, and suppose that $\bar k_0, \bar k_1, \ldots ,\bar k_n$ is an enumeration of $T$ such that for all $s < t \le n$, $\bar k_s \in \rng \bar k_t$. Note that the bijection $\phi$ must send $k_i$ to $\bar k_i$ for all $i\le n$ and preserve all the square relationships. 

We prove by induction that for $\ell_0 = k_0$,  $\ell_i = (k_1 \circ k_2 \circ \cdots \circ k_i)^{-1}(k_0)$, $\bar \ell_0 = \bar k_0$, and $\bar \ell_i = (\bar k_1 \circ \bar k_2 \circ \cdots \circ \bar k_i)^{-1}(\bar k_0)$ where $1 \le i \le n$, that the embedding $\phi^*_i: \Acal_{\ell_i} \to \Acal_{\bar \ell_i}$ induced by $\ell_i \mapsto \bar \ell_i$, sends $k_t \mapsto \bar k_t$ for $t \le i$. 

The base case is obvious. Now suppose it holds for $i < n$. We show it is true for $i+1$. Note that $\ell_{i+1} = k_{i+1}^{-1} \ell_i$ and $\bar \ell_{i+1} = \bar k_{i+1}^{-1}  \bar \ell_i$. Now, $k_{i+1} k_{i+1} \in S$, and hence it is equal to $k_t$ for some $t \le i$, and so $\bar k_{i+1}\bar k_{i+1} = \bar k_t$ as well. Let $u$ be a word in $A_1$ such that $u[\ell_i] = k_t$ (using Lemma \ref{lem:rigidsqrfull}). 

By our induction hypothesis, we have that $u[\bar \ell_i] = \bar k_t$.
Pulling back the statement $u[\ell_i] = k_t$ by $k_{i+1}$  we have $u[\ell_{i+1}] = k_{i+1}$, and pulling back $u[\bar \ell_i] = \bar k_t$ by $\bar k_{i+1}$  we have $u[\bar \ell_{i+1}] = \bar k_{i+1}$. 

Now, note that $\ell_i = k_{i+1} \ell_{i+1} = u[\ell_{i+1}]\ell_{i+1} = w[\ell_{i+1}]$ for a word $w$. Similarly $\bar \ell_i = w[\bar \ell_{i+1}]$. So we have $\phi^*_{i+1}: \ell_i \mapsto \bar \ell_{i}$. And hence using the induction hypothesis on $\phi^*_i$, which agrees with $\phi^*_{i+1}$ on $\{k_0, \ldots, k_i\}$, we have that $\phi^*_{i+1}: k_t \mapsto \bar k_t$ for all $t \le i$. 

 As our induction succeeds, the lemma immediately follows for $ i= n$. 
\end{proof}

\begin{proof}[Proof of Theorem \ref{thm:countablerigidsqrfull}.]
    Let $S \subseteq \Ecal_\lambda$ be countable, rigid, square-confluent, and square-full. We consider $S$ as a tree where we say $k$ is the parent of $j$ if $k$ is the square of $j$. Since $S$ is square-confluent, it is connected as a tree. Let $k_0 \in S$ be an arbitrary element that is fixed from now on. Considering $S$ as a countable connected tree, we can enumerate the vertices as $k_0, k_1, k_2, \dots$, where every $k_n$ is either the square (parent) or a square root (child) of some $k_m$ with $m<n$. Defining $S_n = \{k_0, \dots, k_{n}\}$ for $n <\omega$, we have a sequence of non-empty, finite, rigid, square-full sets $S_n$ such that $S_0 \subseteq S_1 \subseteq S_2 \subseteq \cdots$ and $\bigcup_{n<\omega} S_n = S$. In particular, each $S_{n} = S_{n-1} \cup \{k_{n}\}$ where either $k_{n}$ is the square of some (actually, the largest) element of $S_{n-1}$ or $k_n k_n \in S_{n-1}$.

    We will prove by induction that there are maps $\sigma_n: \la S_n \ra \to \Ccal_1$ such that for $n < m$, $\sigma_m \rest \la S_n \ra = \sigma_n$. Let $\pi_{n,\omega}: \Acal_1 \to \Ccal_1$ be the usual maps into $\Ccal_1$. Since $S_0 =\{k_0\}$, we can define $\rho_0: \Acal_{k_0} \to \Acal_1$ by sending $k_0$ to the generator of $\Acal_1$ and setting $\sigma_0 = \pi_{0,\omega} \circ \rho_0$. 

    We now describe the induction step. If $k_{m+1}$ is the square of some element of $S_m$, then we have $\langle S_{m+1} \rangle = \langle S_m \rangle$ and we set $\sigma_{m+1} = \sigma_m$. 
    
    Now suppose $k_{m+1}$ is the square root of some element of $S_m$. For ease of notation, we re-order $S_m = \{ k_0, k_1, \ldots, k_m\}$ so that $k_{n_1} \in \rng k_{n_2}$ for any $n_1 < n_2$. We also write $k = k_{m+1}$. We have $S_{m+1} = S_m \cup \{k\}$ and $k k = k_i$ for some $i \le m$. Let $r_k \le m$ be largest such that $k_0, \ldots, k_{r_k} \in \rng k$. Let $a_0, \ldots, a_m$ be the images of $k_0, \ldots, k_m$ under $\sigma_m$. By Lemma \ref{lem:conesquarerootsexist} there is an $a \in \Ccal_1$ which is a square root of $a_i$ such that $a$ left divides $a_0, \ldots, a_{r_k}$. Let $\phi(b, b_0, \ldots, b_{r_k})$ be the statement $b$ is a square root of $b_i$ such that $b$ left divides $b_0, \ldots, b_{r_k}$. We have then that $\Ccal_1$ satisfies that there exist some $b$ such that $\phi[b, a_0, \ldots, a_{r_k}]$. (Note that $i \le r_k$). 

    Suppose first $r_k < m$. We have that $a_{r_k+1}$ left-divides $a_0, \ldots, a_{r_k}$. By Corollary \ref{cor:elemembwitness}, there is some $b \in \Ccal_1$, such that $\phi[b, a_0, \ldots, a_{r_k}]$ holds and $a_{r_k+1}$ left-divides $b$. 
    Since $\{a_0, \ldots, a_m\}$ is rigid, iterated applications of Corollary \ref{cor:elemembwitness} show that there is a $b$ such that $\phi[b, a_0, \ldots, a_{r_k}]$ holds and $b$ is left-divided by $a_{r_k+1}, a_{r_k+2}, \ldots, a_m$. Fix such a witness $a^*$, and let $T^* = \{a_0, \ldots, a_m, a^*\}$. 

    For the case that $r_k=m$, we let $a^*$ be any $b$ such that $\phi[b, a_0, \ldots, a_{r_k}]$. We then define $T^*$ in the same way. 

    Let $q < \omega$ be least such that $T^* \subseteq \rng \pi_{q,\omega}$. Fix $j \in \Ecal_{\lambda+1}$. Let $S^*$ be the set of embeddings in $\Acal_j$ corresponding to $\pi_{q,\omega}^{-1}[T^*]$ under the isomorphism $\Acal_j \isom \Acal_1$. We have then that $S^*$ is a finite, rigid, square-full set. Furthermore, there is a natural bijection $S_{m+1} \to S^*$ which preserves squares and the range-ordering. So by Lemma \ref{lem:rigidityuniqueness} there is an induced algebra isomorphism $\tau: \la S_{m+1} \ra \to \la S^* \ra$. We define $\sigma_{m+1}$ to be the composition of $\tau$, the isomorphism $\Acal_j \isom \Acal_1$, and $\pi_{q,\omega}$. It is easy to check that $\sigma_{m+1}$ is as desired and agrees with $\sigma_m$. So the theorem follows by induction. 
\end{proof}

We mention one more consequence of Lemma \ref{lem:rigidityuniqueness}.
In this statement we refer to \emph{rigid} subsets of $\Ccal_1$ by which we mean the same definition as in the case of elementary embeddings, but with $j \in \rng k$ replaced by $a$ left-divides $b$. See Section \ref{section:pullbackalgebra}.

\begin{theorem}
Assume the large cardinal hypothesis: there exists a nontrivial elementary embedding $\vlm \to \vlm$. 
    Suppose that $S, T \subseteq \Ccal_1$ are finite, rigid, and square-full and $\phi: S \to T$ is a bijection which preserves squares and the left-divisor ordering. Then $\phi$ extends to an automorphism of $\Ccal_1$. 
\end{theorem}

\begin{proof}[Proof sketch.]
    Let $S', T' \subseteq \Ecal_\lambda$ be corresponding sets of embeddings such that there are bijections $S \to S'$ and $T \to T'$ preserving squares and the left-divisor ordering.
    The proof of Lemma \ref{lem:rigidityuniqueness} shows that the natural isomorphism which sends the natural pullback of $S'$ to the natural pullback of $T'$ also sends $S'$ to $T'$ pointwise (when ordered according to the range relation). Translating this fact back (as in the proof of Theorem \ref{thm:countablerigidsqrfull}) into $\Ccal_1$ and using Theorem \ref{thm:strongerhomog}, there is an automorphism sending the natural pullback of $S$ to the natural pullback of $T$. And hence this automorphism is as desired.
\end{proof}

\section{Open Questions and Possible Extensions}\label{OpenQsAndExtensions}

Whether or not the large cardinal hypotheses in the above theorems are necessary remains open.

\begin{question}
    Can the conclusions of Theorems \ref{thm:elem_equiv} (1), \ref{lem:elem},  and  \ref{thm:strongerhomog} be proven in ZFC? Or do the conclusions of Theorems \ref{thm:elem_equiv} (1), \ref{lem:elem},  and  \ref{thm:strongerhomog}  imply $\text{Con}(\text{ZFC})$? 
\end{question}

Whether our results on $\Ccal_1$ can be extended to several generators is also unclear.

\begin{question}
    Is there an analogue of $\Ccal_1$ for two (or more) generators? In particular, is there a (simply definable) algebra which embeds $\Acal_2$ and where application is elementary?
\end{question}

\bibliographystyle{plain}
\bibliography{main}

\end{document}